\documentclass[11pt,reqno]{amsart}
\usepackage{amssymb,amsfonts,amsthm,amscd,stmaryrd,dsfont,esint,upgreek,constants,todonotes}
\usepackage[mathcal]{euscript}
\usepackage[latin1]{inputenc}   
\usepackage[notcite, notref]{showkeys}
\usepackage[mathcal]{euscript}
\usepackage{graphicx,color}
\numberwithin{equation}{section}
\newcommand{\N}{\mathbb N}
\newcommand{\Z}{\mathbb Z}

\newcommand{\ol}{\overline}
\newcommand{\oo}{\infty}
\newcommand{\R}{\mathbb R}
\def\E{\mathbb E}
\def\P{\mathbb P}
\newconstantfamily{C}{symbol=C}
\newconstantfamily{c}{symbol=c}
\newconstantfamily{O}{symbol=\Omega}

\def\Xint#1{\mathchoice
{\XXint\displaystyle\textstyle{#1}}%
{\XXint\textstyle\scriptstyle{#1}}%
{\XXint\scriptstyle\scriptscriptstyle{#1}}%
{\XXint\scriptscriptstyle\scriptscriptstyle{#1}}%
\!\int}
\def\XXint#1#2#3{{\setbox0=\hbox{$#1{#2#3}{\int}$}
\vcenter{\hbox{$#2#3$}}\kern-.5\wd0}}

\def\dashint{\Xint-}

\numberwithin{equation}{section}
\newtheorem{thm}{Theorem}[section]
\newtheorem{lem}[thm]{Lemma}

\theoremstyle{definition}

\def\smallnegint{\mathop{\int\mkern-13mu
        \raise.5ex\hbox{${\scriptscriptstyle\diagup}$}}\nolimits}
\def\ds{\displaystyle}

\def\ep{\varepsilon}

\def\div{\operatorname{div}}

\def\ssetminus{\,\raise.4ex\hbox{$\scriptstyle\setminus$}\,}
\def \lg{\langle}
\def \rg{\rangle}
\newcommand{\be}{\begin{equation}}
\newcommand{\ee}{\end{equation}}

\renewcommand{\d}{d}
\newcommand{\Rd}{\mathbb{R}^\d}
\newcommand{\Zd}{\mathbb{Z}^\d}

\renewcommand{\bar}{\overline}
\renewcommand{\tilde}{\widetilde}
\renewcommand{\hat}{\widehat}
\renewcommand{\ln}{\log}

\begin{document}
\title{Scaling limits and stochastic homogenization  for some nonlinear parabolic equations }
\author[Pierre Cardaliaguet, Nicolas Dirr  and Panagiotis E. Souganidis]
{Pierre Cardaliaguet, Nicolas Dirr and Panagiotis E. Souganidis}
\address{Universit\'e Paris-Dauphine, PSL Research University, Ceremade, 
Place du Mar\'echal de Lattre de Tassigny, 75775 Paris cedex 16 - France}
\email{cardaliaguet@ceremade.dauphine.fr }
\address {Cardiff School of Mathematics, Cardiff University, Cardiff, UK}
\email{dirrnp@cardiff.ac.uk}  
\address{Department of Mathematics, University of Chicago, Chicago, Illinois 60637, USA}
\email{souganidis@math.uchicago.edu}
\vskip-0.5in 
\thanks{\hskip-0.149in Pierre Cardaliaguet was partially supported by the ANR (Agence Nationale de la Recherche) project  ANR-12-BS01-0008-01, by the CNRS through the PRC grant 1611 and by the Air Force Office for Scientific Research grant FA9550-18-1-0494. Nicolas Dirr was partially supported by Royal Society via International Exchange Grant  IE160514 - Probabilistic techniques for scaling limits of nonlinear differential equations and by  EPSRC via grant EP/M028607/1. Panagiotis Souganidis was partially supported by the National Science Foundation grants DMS-1600129 and DMS-1900599, the Office for Naval Research grant N000141712095, and the Air Force Office for Scientific Research grant FA9550-18-1-0494.}
\dedicatory{Version: \today}

\begin{abstract}
The aim of this paper is twofold. The first is to study the asymptotics of a parabolically scaled, continuous and space-time stationary in time version of the well-known Funaki-Spohn model in Statistical Physics. After a change of unknowns requiring the existence of a space-time stationary  eternal solution of a stochastically perturbed heat equation, the  problem transforms to the qualitative homogenization of a uniformly elliptic, space-time stationary, divergence form, nonlinear partial differential equation, the study of which  is the second aim of the paper. An important step is the construction of correctors with the appropriate behavior at infinity. 
\end{abstract}

\maketitle      

\begin{section}{Introduction}

The first aim of the paper is to study the limit, as $\ep\to 0$, of the stochastic partial differential equation (SPDE for short) 
\be\label{FS}
\left\{\begin{array}{l}
\ds d_tU^\ep_t= {\rm div} \mathcal{A}(D U^\ep_t, \dfrac{x}{\ep}, \dfrac{t}{\ep^2}, \omega_1)dt+\dfrac{1}{\ep} \sum_{k\in \Z^d}
A(\dfrac{x-k}{\ep}) dB^k_t \ \  {\rm in} \ \ \R^d\times (0,+\infty),\\ 
\ds U^\ep_0=u_0. \ \ \text{in} \ \ \R^d.
\end{array}\right.
\ee

In the above equation,  $(B^k)_{k\in \Z^d}$ is a sequence of independent $d-$dimensional Brownian motions in a probability space $(\Omega_0, \mathcal F_0, \P_0)$ with $\Omega_0= (C^0(\R,\R^d))^{\Z^d}$, and  $A:\R^d  \to \R^d$ is a smooth map with a compact support. Let $(\Omega_1, \mathcal F_1, \P_1)$ be another probability space endowed with a space-time ergodic  group of measure preserving transformations. The  vector field $\mathcal{A}:\R^d \times \R^d \times \R \times \Omega_1 \to \R$ is assumed to be smooth,  uniformly elliptic  and space-time stationary   in $(\Omega_1, \mathcal F_1, \P_1)$, 
and is  independent of the Brownian motions. The precise  assumptions are listed  in section~4. 
\smallskip

A reformulation of \eqref{FS} led us to the second aim of the paper.  This is the study of the qualitative (stochastic) homogenization of the divergence form quasilinear partial differential equation (PDE for short)
\be\label{pde}
u^\ep_t- \div a(Du^\ep, \dfrac{x}{\ep}, \dfrac{t}{\ep^2},\omega) =f \ \ \text{in}  \ \  \R^d\times (0,\infty) \quad u^\ep(\cdot,0)=u_0, 
\ee
where $a:\R^d \times \R^d \times \R \times \Omega \to \R^d$ is strongly monotone, Lipschitz continuous  and space-time stationary in an ergodic with respect to $\Z^d \times \R$-action random environment, which we denote again by  $(\Omega, \mathcal F, \P)$ although it is different than the one for \eqref{FS}, and $f$ and $u_0$ are square integrable. All the assumptions are made precise in section~3.

\smallskip

The result is that, in either case, there exists a strongly monotone map $\ol a:\R^d\to \R^d$ such that the solutions of \eqref{FS} and \eqref{pde} converge either a.s. or in expectation and in an appropriately weighted $L^2$-space in space-time to the unique solution $\ol u$ of the initial value problem
\be\label{ivp}
\ol u_t-\div \ol a(D\ol u)=f    \ \ \text{in}  \ \  \R^d\times (0,\infty) \quad \ol u(\cdot,0)=u_0.
\ee
\smallskip

The link between \eqref{FS} and \eqref{pde} is made writing $U^\ep$ as 
$$
U^\ep_t(x)= \ep \tilde V_{\frac{t}{\ep^{2}}}(\dfrac{x}{\ep}) +  \tilde W^\ep_t(x),
$$
with $\tilde V$ and $W^\ep$ been respectively the unique up to constants  eternal, space-time stationary solution of the stochastically perturbed heat equation
\be\label{l2}
d\tilde V_t= \Delta \tilde V_tdt+ \sum_{k\in \Z^d} A(x-k) dB^k_t \ \ \text{in} \ \ \R^d \times \R,
\ee
and 
the solution of the uniformly elliptic, divergence form PDE
\be\label{intro.homopb}
 \partial_t \tilde W^\ep_t =   {\rm div} \Big(\tilde a(D\tilde W^\ep, \dfrac{x}{\ep},\dfrac{t}{\ep^{2}},\omega)\Big) \ \  {\rm in} \ \ \R^d\times (0,+\infty) \quad  \tilde W^\ep_0= u_0 \ \ \text{in} \ \ \R^d,
\ee
with the random nonlinearity 
\be\label{m112}
\tilde a(p, x ,t, \omega)= \mathcal A(p+D\tilde V_t(x, \omega_0),  t,x, \omega_1) - D\tilde V_t(x,\omega_0) 
\ee
space-time stationary, strongly monotone and Lipschitz continuous.   
\smallskip

The existence and properties of $\tilde V$ are the topic of section~2. The construction is based on solving  the problem in $\R^d \times [-n^{-2},\oo)$ and then letting $n\to \oo$. To prove, however, the convergence to a unique up to constants stationary solution, it is necessary to obtain suitable gradient bounds. This requires, among others, the quantitative  understanding the long-space decorrelation properties of the gradients. For the latter, it is necessary to study in detail the properties of the gradients of localized versions of the stochastically perturbed heat equation, which depend on finitely many Brownian motions in balls of radius $R$, as $R\to \oo$. 
\smallskip

The study of the qualitative homogenization of \eqref{pde}, which is developed in  section~\ref{sec.homogen},  is based on the existence,  for each $p\in\R^d$, of space-time stationary solutions $\chi^p=\chi(y,\tau,\omega;p)$ of 
\be\label{eq.thetaw.1}
\partial_\tau \chi^p -{\rm div}( a(p+D\chi^p,y,\tau,\omega))= 0 \ \ \text{in} \ \ \R^d,  
\ee
such that,  as $\ep\to 0$, $\chi^\ep(x,t;p,\omega)=\ep  \chi^p(\frac{x}{\ep},\frac{t}{\ep^2},\omega) \to 0$  in $L^2_{loc}(\R^{d+1})$, 
$\P-$a.s. and in expectation.
\smallskip

The existence of correctors in our setting is, to the best of our knowledge, new. The difficulty arises from the unbounded domain and the lack of regularity in time. Overcoming it,  requires the development of new and sharp results. 
\smallskip

Once correctors are established, the homogenization follows, at least formally, using, at the level of test functions $\phi$, the expansion
\[ \phi^\ep(x,t)=\phi(x,t) +\ep \chi(\dfrac{x}{\ep}, \dfrac{t}{\ep^2}, \omega;D\phi(x,t)), \]
the justification of which creates additional problems due to the low available regularity of $\chi^p$ in $p$. To overcome it, it is necessary 
to introduce yet another level of approximation involving ``piecewise gradient correctors'' corresponding to piecewise constant approximations  of $D\phi$. 

\bigskip

Funaki and Spohn showed  in \cite{FS} the convergence of a system of interacting diffusion processes, modeling the height of a surface in $\R^d$,  to a deterministic limit equation.
\smallskip
 
More precisely, for any cube $\Lambda\subseteq \Zd$, they considered processes of the form
\be\label{fs1}
d\Phi_t(x)=-\sum_{|x-y|_1=1}V'(\Phi_t(x)-\Phi_t(y))+\sqrt{2}dB_t(x) \ \ \text{for} \ \  x\in \Lambda\subset \Zd . 
\ee
The fields $\Phi$ live on a discrete lattice and take values in $\R^d$,  $B_t(x)$ are i.i.d. Brownian motions,
$V'$ is the derivative of a strictly convex 
symmetric function, and $|\cdot |_1$ is the $l^1-$norm. Note that 
the drift term in \eqref{fs1} is simply the discrete divergence of the vector field
$(V'(D^+_i \Phi))_{i=1,\ldots,d},$ where $D^+_i \Phi(x)=\Phi(x+e_j)-\Phi(x)$ is the discrete forward partial derivative in direction $i$. 
\smallskip

The result in \cite{FS} is that the rescaled fields 
\be\label{fs2}
\Phi^\varepsilon(r,t)=\varepsilon\Phi_{\varepsilon^{-2}t}(x)\ \  {\rm for}\ \ r\in [x-\varepsilon/2,x+\varepsilon/2)^d \ \ \text{with} \ \ 
 N=[\varepsilon^{-1}]
 \ee
converge to the solution $h$ of the  nonlinear, divergence form deterministic PDE
\[
\partial_t h(r,t)={\rm div}(D \sigma(\nabla h)) \ \ \text{in} \ \ \R^d \times (0,\infty).\]

A crucial  step in the proof in  \cite{FS} is the existence of unique gradient Gibbs measures, that is, invariant measures for the discrete gradient of the fields which on finite subsets $\Lambda\subseteq \Zd$ defined by
$$
\frac{1}{Z}e^{\beta \sum_{x\in \Lambda }V(\nabla_i\Phi(x))} \prod_{x\in \Lambda}d\Phi(x).
$$

The SPDE \eqref{FS} we are considering here can be seen as   a continuous version of the equation satisfied by $\Phi^\varepsilon$ in  \eqref{fs2}. Our proof of the existence of the limit is purely dynamic, that is, it  does not use the existence of invariant measures of a certain structure. Instead, we use the  eternal solutions of a linear SPDE, which allows to transform the problem to one like \eqref{pde} with an appropriately defined field $a$. 
\smallskip

Although it may appear so,  results  about the convergence of the solution of $U^\ep_t$ and  $\Phi^\varepsilon$ 
are not, in any sense, equivalent. For example, the effective nonlinearities $\ol a$ and $h$ are, in general, not the same. To be able to compare the limit problems, it is necessary to understand in precise way how \eqref{FS} with $\ep=1$  is the continuous (mescopic) limit of \eqref{fs1}.
\smallskip

The qualitative stochastic homogenization result is new. We are, of course, aware of earlier  works of Efentiev and Panov \cite{EP1, EP2} and Efendiev, Jiang and Pankov \cite{EfJiPa}, which, however, do not apply to the general space-time stationary setting we are considering here. The crucial part of the proof is the existence of a space-time stationary corrector, which requires overcoming the low regularity in time. Beside the references \cite{EfJiPa, EP1, EP2} already quoted, the literature on the space-time homogenization of  parabolic equations in a random setting  is scarce and mostly devoted to linear equations, starting with the pioneering work of Zhikov, Kozlov and Oleinik \cite{ZKO}:  Landim, Olla and Yau \cite{LOY} provide an invariance invariance principle for diffusion in space-time random environment with a bounded  stream matrix;  Fannjiang and Komorowski \cite{FaKo} generalize the result to the case of unbounded stationary vector potentials while Komorowski and Olla \cite{KO} investigate the problem for divergence free vector fields; Rhodes \cite{Rh07} and Delarue and Rhodes \cite{DeRh09} study the homogenization of degenerate diffusions; more recently, Armstrong, Bordas and Mourrat \cite{ABM} provide a convergence rate for the homogenization of parabolic equation in space-time random environment under a finite range condition by using a variational structure for the equation. 

\subsection*{Organization of the paper} Section~2 is devoted to the study of the linear problem \eqref{l2}. In section~3 we concentrate on \eqref{pde}. The result about \eqref{FS} is presented in section~\ref{sec.homoFS}. Each of  section  consists of several subsections which are outlined there.  Finally, in the Appendix we include  some results about functions with stationary gradients that we use throughout the  paper. 



\subsection*{Notation}  Given $x_0\in \R^d$,  $Q_R(x_0)= x_0 + (-R/2,R/2)^d$ and $B_r(x_0)$ is the open ball in $\R^d$ centered at $x_0$ and radius $r$. Moreover,  $Q_R=Q_R(0)$, $I_R= (-R/2,R/2)$, $B_r=B_r(0)$,  and $\tilde Q_R= Q_R\times I_R=(-R/2,R/2)^{d+1} \subset \R^{d+1}$, while $\tilde Q$ and $\hat Q$ are  used for any cube in $R^{d+1}$. If $a, b\in \R$, $a\wedge b=\min\{a, b\}$ and  $a\lesssim_{\alpha} b$  means  that there exists a constant $C=C(\alpha)>0$ such that $a\leq C b$.  We write $a \underset{\alpha}\sim b$ if $a\lesssim_{\alpha} b$ and $b\lesssim_{\alpha} a$.  The  integer part of $s\in \R$ is $\lfloor s \rfloor$.  Given  $x\in \R^d$, $|x|_\infty= \max\{|x_i|: i=1, \dots, d\}$. We write $ {\bf 1}_A$ for the characteristic function of a set $A$ and, finally,  $\text{Int B}$ is the topological interior of $B\subset \R^k$.

\subsection*{Terminology} We say that a vector field $b:\R^d \to \Rd$ is strongly monotone and Lipschitz continuous  if the there exists $C_0>0$ such that, respectively and for all $p,q \in \R^d$, 
\be\label{monotone}
 (b(p)-b(q))\cdot (p-q) \geq C_0^{-1}|p-q|^2,
\ee
and
\be\label{Lip}
 |b(p)-b(q)|\leq C_0|p-q|.
 \ee

\begin{section} {The linear problem \eqref{l2}}\label{sec:linearpb}

The  goal here is to construct space-time stationary solutions of  the linear SPDE
\be\label{eq:main}
dV_t= \Delta V_tdt+ \sum_{k\in \Z^d} A(x-k) dB^k_t.
\ee

A building block is the properties of the solutions of the initial value problem
\be\label{l1}
\left\{\begin{array}{l}
\ds dV_t = \Delta V_t dt+  \sum_{k\in \Z^d} A(x-k) dB^k_t \ \ {\rm in} \ \  \R^d\times (0,+\infty),\\ 
\ds V_0=0 \ \  {\rm in } \ \  \R^d,
\end{array}\right.
\ee
since the solution of \eqref{eq:main} is going to be obtained as the limit of solutions of 
\be\label{ln}
\left\{\begin{array}{l}
\ds dV^n_t(x)= \Delta V^n_t(x)dt+ \sum_{k\in \Z^d} A(x-k) dB^k_t \ \ \text{in} \ \ \R^d\times (-n^2, \oo), \\ 
\ds V^n_{-n^2}=0 \ \  {\rm in } \ \  \R^d.
\end{array}\right.
\ee
It is immediate that $V^n$ satisfy bounds similar to the ones of the solution of \eqref{l1}.

\smallskip

We divide the presentation into a a number of subsections. In subsection~\ref{ass1} we introduce the assumptions we need to study the problem and state the result. In subsection~\ref{aux1} we prove a number of basic estimates for the solution of \eqref{l1}. These estimates are not sufficiently strong in order to let $n\to \oo$ in \eqref{ln}. In subsection~\ref{dec} we obtain some new stronger estimates  taking advantage of the independence at large distances of the Brownian motions.    The proof of Theorem~\ref{thm:main} is presented in  subsection~\ref{proof1}.
\smallskip

\subsection{The assumptions and result}\label{ass1}
We assume that
\be\label{B} 
\begin{cases} \text{The family $(B_k)_{k\in \Z}$  consists  of continuous and independent  $d-$dimensional }\\
\text{ processes defined on the probability space $(\Omega_0, \mathcal F_0, \P_0)$ with}\\
\text{ $\Omega_0=(C(\R;\R^d))^{\Z^d}$ \ such that, for any $t_0\in \R$, $(B^k_t-B^k_{t_0})_{t\geq t_0}$}\\
\text{ is a Brownian motion,} 
\end{cases}
\ee
and 
\be\label{A}
\begin{cases}
\text{ the map $A:\R^d\to \R^d$ is smooth and has compact support}\\
\text{ in the ball $B_{R_0}$ for some $R_0>0$}.
\end{cases}
\ee
The assumptions  on $A$ are made for simplicity and can be relaxed. Moreover, since the coefficients of the noise in \eqref{eq:main} are deterministic, the question of whether we need to use It\^{o}'s or Stratonovich stochastic differential  does not arise here.
\smallskip

In the context of \eqref{eq:main},   a process is stationary,  if it is  adapted to the filtration generated by the $(B^k)_{k\in \Z^d}$ with a law which is invariant by translation in time and integer translation in space.  
\smallskip

The existence of a unique up to constants stationary solution of \eqref{eq:main} is the subject of the next theorem. In what follows by solution we mean a map $Z:\R^d\times\R \times \Omega_0 \to \R$ such that, for any $x \in \R^d$, $s,t \in \R$ with $s<t$ and $\P-$ a.s. $\omega_0 \in \Omega_0$, 
\be\label{m115}
Z_t(x, \omega_0)= Z_s(x, \omega_0) + \sum_{k\in \Z^d}  \int_0^t \int_{\R^d} p(x-y,t-r)  A(y-k) dy dB^k_r(\omega_0),  
\ee
where $p=p(x,t)$ the heat kernel, that is, the fundamental solution to the heat equation in $\R^d\times (0,\oo)$.
\smallskip

\begin{thm}\label{thm:main}
Assume \eqref{B} and \eqref{A}. There exists a unique  process $Z:\Omega_0\times \R^d\times \R\to \R^d$ with
$$
{\mathbb E}\left[\int_{ \tilde Q_1}|Z_t(x)|^2dxdt\right]<\infty,
$$ 
satisfying, for any $i=1, \dots, d$,  
$$
dZ_{i,t}(x) = \Delta Z_{i,t}(x)dt+    \sum_{k\in \Z^d} D_{x_i}A(x-k) dB^k_t.
$$
In addition, $Z$ is an attractor for \eqref{eq:main} in the sense that, if $V$ is a solution of \eqref{eq:main} in $\R^d\times (0,\oo)$ such that $V(\cdot,0)=0$, 
then 
\be\label{attractor}
\lim_{t\to+\infty} \E\left[ \int_{Q_1} |DV_t(x)- Z_t(x)|^2dx \right] =0. 
\ee
Moreover, for $d\ge 3,$  there exists a unique up to constants space-time stationary adapted process $V:\Omega\times \R^d\times \R\to \R$ solving  \eqref{eq:main} in $\R^d\times \R$ 
such that  
$$
{\mathbb E}\left[\int_{ \tilde Q_1}|V_t(x)|^2dxdt\right]<\infty.
$$ 
\end{thm}

We remark that, when $d\leq 2$, the correctors have  stationary gradients  but are not themselves stationary. 

\subsection{Auxiliary results}\label{aux1} We concentrate here on the properties of the solutions of the auxiliary initial value problem. 
\smallskip

%
%
%

%
The first result is about a representation formula for the solution of \eqref{l1} as well as preliminary integral bounds on its derivatives. 
\smallskip

Note that  the forcing term in \eqref{l1} is periodic only in law and not pointwise. Hence, all the estimates need involve expectation.

\begin{lem}\label{lem.boundDV} Assume \eqref{B} and \eqref{A}. Then 
\begin{equation}\label{rep.sol}
V_t(x)= \sum_{k\in \Z^d}  \int_0^t \int_{\R^d} p(x-y,t-s)  A(y-k) dy dB^k_s  
\end{equation}
is a stationary in space with respect to integer  translations solutions,  
solution $V$ of \eqref{l1}. Moreover, for all $t\geq 0$, 
\begin{equation}\label{ineq.timecont}
\sup_{x\in \R^d} \E[|D V_t(x)|^2]+\E[|D^2 V_t(x)|^2] \lesssim_{A,d} (t\wedge 1), 
\end{equation}
and
\be\label{ineq.timecontBIS}
\E[|V_t(x)|^2] \lesssim_{A,d}  t\wedge  \left\{\begin{array}{ll}  1 & {\rm if}\; d\geq 3,\\[1.2mm]  \ln(t+1) &  {\rm if}\; d=2,\\[1.2mm]
t^{1/2} & {\rm if}\; d=1.\end{array}\right.  
\ee 
\end{lem}
\begin{proof} It is immediate that the $V$ given in \eqref{rep.sol} satisfies \eqref{m115}  for any $0<s<t$ and, hence, is a solution of \eqref{l1}.   It also  follows from \eqref{rep.sol}  and the fact that the $B^k$'s  are identically distributed that $V$ is stationary in space under integer translations. Hence,  we only  need to prove the estimates for $x\in Q_1$. 
\smallskip

It\^{o}'s isometry and \eqref{rep.sol} yield  
\begin{align*}
\E\left[ |D V_t(x)|^2\right] & = \sum_{k\in \Z^d}    \int_0^t \left| \int_{\R^d} D p(x-y,t-s)  A(y-k) dy\right|^2 ds .
\end{align*}
For $k\in \Z$ and $s\geq 0$, let 
\be\label{Fk}
F_k(s)= \left| \int_{\R^d} D p(x-y,t-s)  A(y-k) dy\right|= \left| \int_{\R^d}  p(x-y,t-s)  D A(y-k) dy\right|.
\ee

To proceed we need the following lemma. Its proof is presented after the end of the ongoing one. 

\begin{lem}\label{lem.intermezzo}
Assume \eqref{B} and \eqref{A} and, for $k\in \Z$ and $s\geq 0$, let $F_k(s)$ be given by \eqref{Fk}. Then
\begin{eqnarray}
 \sum_{k\in \Z^d} F_k(s)^2 &\lesssim_{A,d}& (t-s)^{-(1+d/2)},  \label{(a)}\\
\sum_{k\in \Z^d, \, |k|\geq 2(R_0+2)} F_k(s)^2  &\lesssim_{A,d}&(t-s)^{-1-d/2}\exp(-4R_0^2/(17(t-s))),  \label{(b)}\\
F_k(s)^2  &\leq&\|D A\|_\infty^2.  \label{(c)}
\end{eqnarray}
\end{lem}

We continue with the proof of Lemma~\ref{lem.boundDV}. 
\smallskip

The arguments depend on whether $t\geq 1$ or $t<1$.
\smallskip

If $t\geq 1$, we observe that  there are only finitely many $k$ with 
$ |k|< 2(R_0+2)$ and we find, using Lemma~\ref{lem.intermezzo}, that 
\begin{align*}
&\E\left[ |D V_t(x)|^2\right]  \\
&   \leq  \sum_{k\in \Z^d} \int_0^{t-1} F_k(s)^2ds + \sum_{k\in \Z^d, \ |k|\geq 2(R_0+2) } \int_{t-1}^t F_k(s)^2ds
+ \sum_{k\in \Z^d, \ |k|< 2(R_0+2) }\int_{t-1}^t F_k(s)^2ds  \\[1.5mm]
& \lesssim_{A,d}  \left(\int_0^{t-1} (t-s)^{-(1+d/2)}ds+ \int_{t-1}^t (t-s)^{-1-d/2}\exp\{-4R_0^2/(17(t-s))\}ds\right) \\
& \qquad \qquad +  \sum_{k\in \Z^d, \ |k|< 2(R_0+2) }\int_{t-1}^t \|D A\|_\infty^2ds \\ 
& \lesssim_{A,d} \left(  1 + \int_0^1 s^{-1-d/2}\exp\{-4R_0^2/(17s)\}ds\right) \lesssim_{A,d} 1. 
\end{align*}
If $t\in (0,1]$,  using  \eqref{(b)} and \eqref{(c)}, we obtain
\begin{align*}
\E\left[ |D V_t(x)|^2\right] &  \leq  \sum_{k\in \Z^d, \ |k|\geq 2(R_0+2) } \int_{0}^t F_k(s)^2ds
+ \sum_{k\in \Z^d, \ |k|< 2(R_0+2) }\int_{0}^t F_k(s)^2ds \\
& \lesssim_{A,d}  \left( \int_0^t (t-s)^{-1-d/2}\exp\{-4R_0^2/(17(t-s))\}ds + t\right) \\ 
& \lesssim_{A,d}  \left(\int_0^t s^{-1-d/2}\exp\{-4R_0^2/(17s)\}ds + t\right) \lesssim_{A,d}  t. 
\end{align*}
Since the structure of the formula  for $D^2V$ is exactly the same as the one for $DV$,  \eqref{ineq.timecont} is proved similary. The only difference is that now  the  constants depend  on $\|A\|_{C^2}$ too.
\smallskip

To estimate $V_t(x)$, recalling that,  for any $x,s,t$,
$\int_{\R^d} p(x-y,t-s)dy = 1,$
we find 
\begin{align*}
& \E\left[ | V_t(x)|^2\right]  \leq \|A\|_\infty^2  \sum_{k\in \Z^d}    \int_0^{t} \left| \int_{B_{R_0+2}(k)} p(x-y,t-s) dy\right|^2 ds \\
& \qquad \lesssim_{A,d}  \big(\sum_{k\in \Z^d}    \int_0^{t-1} \int_{B_{R_0+2}(k)} p^2(x-y,t-s) dy ds\\
& \qquad \qquad + 
\int_{(t-1)\vee 0}^t  \sum_{|k|\geq R_0+3} \int_{B_{R_0+2}(k)} p^2(x-y,t-s)dy \ ds
+  (t- (t-1)\vee 0) \big).
\end{align*}
The first term in the right-hand side can be estimated by 
\begin{align*}
& \int_0^{(t-1)\vee 0}\sum_{k\in \Z^d} \int_{B_{R_0+2}(k)} p^2(x-y,t-s)dy ds\lesssim_{R_0} \int_0^{(t-1)\vee 0}\int_{\R^d} p^2(x-y,t-s)dyds\\
& \qquad \lesssim_{R_0}\int_0^{(t-1)\vee 0} (t-s)^{-d/2}ds \lesssim_{R_0,d}  {\bf 1}_{t\geq 1}  \left\{\begin{array}{ll}  1 & {\rm if} \ \ d\geq 3,\\ \ln(t+1)&  {\rm if} \ \ d=2,\\
t^{1/2} & {\rm if} \ \  d=1.\end{array}\right.
\end{align*}
As for second term in the right-hand side, we have 
\begin{align*}
 \int_{(t-1)\vee 0}^t & \sum_{|k|\geq R_0+3} \int_{B_{R_0+2}(k)} p^2(x-y,t-s)dy \ ds\\
 & \lesssim_{R_0}  \int_{(t-1)\vee 0}^t  \int_{B_1^c} p^2(x-y,t-s)dy \ ds \\
&  \lesssim_{R_0}  \int_{(t-1)\vee 0}^t (t-s)^{-d} \int_1^{+\infty} r^{d-1} \exp\{-r^2/(t-s)\}  dr ds \\
& \lesssim_{R_0,d} \int_{(t-1)\vee 0}^t (t-s)^{1-d} \exp\{-1/(2(t-s))\} ds \lesssim_{R_0,d} (t\wedge 1).
\end{align*}
The proof of  \eqref{ineq.timecontBIS} is now complete. 

\end{proof}

We present  now the proof of Lemma~\ref{lem.intermezzo}.
\begin{proof}[Proof of Lemma \ref{lem.intermezzo}]
It follows from \eqref{A} and the fact that $x\in Q_1\subset B_2$ that 
\begin{align*}
F_k(s)\leq  & \int_{\R^d} \left| D p(x-y,t-s)  A(y-k) \right| dy \leq \|A\|_\infty \int_{B_{R_0+2}} \left| D p(k-y,t-s)   \right| dy \\
 &  \lesssim_{A} (t-s)^{-1-d/2}   \int_{B_{R_0+2}} |k-y| \exp\{-|k-y|^2/(2(t-s))\} dy .
 \end{align*}
 If   $|k|\geq 2(R_0+2)$, then, for any $y\in  B_{R_0+2}$, we have 
\begin{align*} 
 |k-y| \exp\{-&|k-y|^2/(2(t-s))\}\\
 & \leq  (|k|+(R_0+2)) \exp\{ -|k|^2/(4(t-s)) +(R_0+2)^2/ (2(t-s))\}\\
 & \leq  (|k|+(R_0+2)) \exp\{ -|k|^2/(8(t-s))\}
\\
 & \lesssim_{A}  |k| \exp\{ -|k|^2/(16(t-s))\}.
 \end{align*}
Thus,
\[
F_k(s) \lesssim_{A,d}  \begin{cases}
 (t-s)^{-1-d/2}  |k| \exp\{-|k|^2/(16(t-s))\} \ \  {\rm  if}\ \  |k|\ge 2(R_0+2),\\[1.5mm] 
 (t-s)^{-1-d/2} \  {\rm  if}\  |k|\le 2(R_0+2).
\end{cases}
\]
Then \eqref{(a)} follows, since
\begin{align*}
&  \sum_{k\in \Z^d} F_k(s)^2  \lesssim_{A,d}    (t-s)^{-2-d}(1+ \sum_{k\in \Z^d}   |k|^2 \exp\{-|k|^2/(8(t-s))\}) \\[1.5mm]
& \lesssim_{A,d}  (t-s)^{-2-d}(1+ \int_{\R^d}   |z|^2 \exp\{-|z|^2/(16(t-s))\}dz )\lesssim_{A,d}  (t-s)^{-(1+d/2)}. 
 \end{align*}
 
For \eqref{(b)}, using that, for all $r\geq 0$, $r^{d+1} \exp\{-r^2/16\} \lesssim r\exp\{-r^2/17\}$,  we get 
\begin{align*}
&  \sum_{k\in \Z^d, \, |k|\geq 2(R_0+2)} F_k(s)^2  \lesssim_{A,d}    (t-s)^{-2-d}\sum_{k\in \Z^d, \, |k|\geq  2(R_0+2)}   |k|^2 \exp\{-|k|^2/(8(t-s))\} \\ 
& \lesssim_{A,d}   (t-s)^{-2-d}\int_{B_{ 2R_0}^{c}}   |z|^2 \exp\{-|z|^2/(16(t-s))\}dz, \\
& \lesssim_{A,d}   (t-s)^{-1-d/2}\int_{ 2R_0(t-s)^{-1/2}}^{+\infty} r^{d+1} \exp\{-r^2/16\}dr, \\
& \lesssim_{A,d}   (t-s)^{-1-d/2}\exp(-4R_0^2/(17(t-s))).
 \end{align*}

Finally, \eqref{(c)} is straightforward, since 
\begin{align*}
F_k(s) & \leq \|D A\|_\infty \int_{\R^d} p(x-y,t-s)dy = \|D A\|_\infty. 
 \end{align*}

 \end{proof}
\subsection{The decorrelation estimates}\label{dec}
We show that the solution $V$ of \eqref{l1} decorrelates in space. 
\smallskip

To quantify  this property,  we consider solutions of a localized versions of \eqref{l1}, that is, problems  that depend only on the Brownian motions in a certain ball. 
\smallskip

For  $l \in \Z$ and $R\geq 1$, let  $V^{l,R}$ be the solution to 
\be\label{eq:mainLinR}
\left\{\begin{array}{l}
\ds dV^{l,R}_t= \Delta V^{l,R}_tdt + \sum_{k\in \Z^d, \ |k-l|\leq R} A(x-k) dB^k_t \ \  {\rm in}  \ \ \R^d\times (0,+\infty),\\ 
\ds V^{l,R}_0=0 \ \ \text{in} \ \ \R^d.
\end{array}\right.
\ee

\begin{lem}\label{lem.VVlR} Assume \eqref{B} and \eqref{A} and let $V$ be the solution to \eqref{l1}. Then there exists $R_1>0$ such that, for any $R\geq R_1$, $l\in \Z^d$ and $x\in Q_1(l)$,  
\begin{equation}\label{liqendsfdxgc}
\E\left[|DV_t(x)-DV^{l,R}_t(x)|^2\right] \lesssim_{A,d}\left\{\begin{array}{ll}
 R^{-d} &{\rm if } \; R^2/t\leq 1,\\[1.5mm] 
 \exp\{-R^2/(5t)\} &{\rm otherwise,} \end{array},\right.
\end{equation}
and
\begin{equation}\label{lkjnefjjn}
\sup_{t\geq 0} \E\left[|DV^{l,R}_t(x)|^2\right]  \lesssim_{A,d} 1.
\end{equation}

If $d\geq 3$, then, for all $R\geq R_1$,  
\begin{equation}\label{liqendsfdxgcBIS}
\E\left[\left(V_t(x)-V^{l,R}_t(x)\right)^2\right] \lesssim_{A,d} \left\{\begin{array}{ll}
 R^{2-d} &{\rm if } \; R^2/t\leq 1,\\[1.5mm] 
 \exp\{-R^2/(9t)\} &{\rm otherwise}, \end{array}\right.
\end{equation}
and
\begin{equation}\label{lkjnefjjnBIS}
\sup_{t\geq 0} \E\left[(V^{l,R}_t(x))^2\right]  \lesssim 1.
\end{equation}
\end{lem}
\smallskip

For later use, we note that $V^{R,l}_t(x)$ and $V^{R, l'}_{t'}(x')$ are independent, for any $t,t'$ and $x,x'$,  as soon as $|l-l'|> 2R$.
For this reason, we consider  Lemma \ref{lem.VVlR} as a decorrelation property of $V$. 

\begin{proof} Using  the representation formulae of  $DV_t$ and $DV^{l,R}_t$, we find 
$$
D(V_t-V^{l,R})(x)= \sum_{|k-l|>R}  \int_0^t \int_{\R^d} Dp(x-y,t-s) A(y+k)dy dB^k_s.
$$
Then \eqref{A}, It\^{o}'s isometry and Cauchy-Schwartz inequality yield
\begin{align*} 
& \E\left[ |D(V_t-V^{l,R})(x)|^2\right] \leq \sum_{|k-l|>R } \int_0^t \left| \int_{\R^d} Dp(x-y,t-s) A(y+k)dy\right|^2ds \\
& \lesssim_{A}   \sum_{|k-l|>R } \int_0^t \left| \int_{B_{R_0}(k)} (t-s)^{-1-d/2}|x-y| \exp\{-|x-y|^2/(2(t-s))\} dy\right|^2ds \\
& \lesssim_{A}    \sum_{|k-l|>R } \int_0^t  \int_{B_{R_0}(k)} (t-s)^{-2-d} |x-y|^2 \exp\{-|x-y|^2/(t-s)\} dyds.
\end{align*}
Therefore, for $x\in Q_1(l)$,  we get
\begin{align*} 
& \E\left[ |D(V_t-V^{l,R})(x)|^2\right] \lesssim_{A}    \int_0^t  \int_{B_{((R-R_0)-1)}^c(0) } (t-s)^{-2-d} |y|^2 \exp\{-|y|^2/(2(t-s))\} dyds\\ 
& \lesssim_{A}    \int_0^t  \int_{((R-R_0)-1)(t-s)^{-1/2}}^{+\infty}  (t-s)^{-1-d/2} \rho^{d+1} \exp\{-\rho^2/2\} d\rho ds.
\end{align*}
Choosing $R$ large enough, we can assume that $(R-R_0-1)\geq R/2$, 
and using that, for $\rho\geq 0$,  $\rho^{d+1} \exp\{-\rho^2/2\} \lesssim  \rho\exp\{-\rho^2/4\}$, integrating in space and an elementary change of variables we find   
\begin{align*} 
& \E\left[ |D(V_t-V^{l,R})(x)|^2\right]   \lesssim_{A} \int_0^t  (t-s)^{-1-d/2} \exp\{-R^2/(4(t-s))\} ds\\
&  \lesssim_{A} R^{-d} \int_{R^2/t}^{+\infty}  \tau^{-1+d/2} \exp\{-\tau/4\} d\tau,
\end{align*}
and, hence, \eqref{liqendsfdxgc}.
\smallskip

The proof of \eqref{lkjnefjjn} is then follows using  \eqref{liqendsfdxgc} combined and  Lemma~\ref{lem.boundDV}. \\

Next we assume that $d\geq 3$. Then 
$$
(V_t-V^{l,R})(x)= \sum_{|k-l|>R}  \int_0^t \int_{\R^d} p(x-y,t-s) A(y+k)dy dB^k_s,
$$
and, again, \eqref{A}, It\^{o}'s isometry  and an application of the Cauchy-Schwartz inequality imply that 
\begin{align*} 
& \E\left[ ((V_t-V^{l,R})(x))^2\right] \leq \sum_{|k-l|>R } \int_0^t \left| \int_{B_{R_0+2}(k)} p(x-y,t-s) A(y+k)dy\right|^2ds \\
&  \lesssim_{A} \sum_{|k-l|>R } \int_0^t \int_{B_{R_0+2}(k)} (t-s)^{-d} \exp\{-|x-y|^2/(t-s)\} dyds.
\end{align*}
Therefore, if  $x\in Q_1(l)$,  
\begin{align*} 
& \E\left[ ((V_t-V^{l,R})(x))^2\right] \lesssim_{A}  \int_0^t  \int_{B_{((R-R_0)-1)}^c(0) } (t-s)^{-d}  \exp\{-|y|^2/(t-s)\} dyds\\ 
&\lesssim_{A}  \int_0^t  \int_{(R-R_0-3)(t-s)^{-1/2}}^{+\infty}  (t-s)^{-d/2} \rho^{d-1} \exp\{-\rho^2\} d\rho ds.
\end{align*}
Assuming that $R$ is large so that $(R-R_0)-1\geq R/2$,  
using that,  $\rho\geq 0$,  $\rho^{d-1} \exp\{-\rho^2\} \lesssim \rho\exp\{-\rho^2/2\}$ and integrating in space, we get 
\begin{align*} 
& \E\left[ ((V_t-V^{l,R})(x))^2\right] \leq  C  \int_0^t  (t-s)^{-d/2} \exp\{-R^2/(8(t-s))\} ds\\
& \qquad \leq C R^{2-d} \int_{R^2/t}^{+\infty}  \tau^{-2+d/2} \exp\{-\tau/8\} d\tau.
\end{align*}
Then \eqref{liqendsfdxgcBIS} follows easily and the proof of \eqref{lkjnefjjnBIS} is then an application of \eqref{liqendsfdxgcBIS} combined with Lemma~\ref{lem.boundDV}.

\end{proof}

\subsection{The proof of Theorem~\ref{thm:main}}\label{proof1}
To prove the existence of a stationary solution of \eqref{l2}, we consider the sequence of solutions $V^n$ of \eqref{ln}.
\smallskip

The main step is to show that  $(DV^n)_{n\in \N}$ is a Cauchy sequence. 

%
\begin{lem}\label{lem.Cauchy} Assume \eqref{B} and \eqref{A}. Then, for any $r>0$ and any $T>0$, the sequence   $(DV^n)_{n\in \N}$ is  Cauchy  in $L^2( B_r\times [-T,T] \times \Omega)$, that is,  for any $n, m\in \N$ and $t\in \R$ with $m>n$ and $t\in [-n^2+1, m-1]$, 
\begin{align*}
& \E\left[|D(V_t^m-V_t^n)(x)|^2\right] \lesssim_{A,d} C(t+n^2)^{-(1\wedge (d/4))}.
\end{align*}
\end{lem}

\begin{proof} Fix $n<m$ and $t\in [ -n^2,m-1]$. Since $V^m-V^n$ solves the heat equation on $[-n^2, t]$ with initial condition $V^m_{-n^2}$,  we have 
\begin{align*}
V_t^m(x)-V_t^n(x) = \int_{\R^d} p(x-y,t+n^2) V^m_{-n^2}(y)dy .
\end{align*}
Hence, 
\begin{align*}
& \E\left[|D(V_t^m-V_t^n)(x)|^2\right] = \int_{\R^{2d}} p(x-y,t+n^2)p(x-y',t+n^2) \E\left[DV^m_{-n^2}(y)\cdot DV^m_{-n^2}(y')\right]dydy' \\
& \qquad =  \sum_{k,k'\in \Z^d} \int_{Q_1(k)\times Q_1(k')} p(x-y,t+n^2)p(x-y',t+n^2) \E\left[DV^m_{-n^2}(y)\cdot DV^m_{-n^2}(y')\right]dydy'.
\end{align*}
Fix $R= \lfloor(t+1+n^2)^{1/4}\rfloor$,  
and consider, for  $l\in \Z^d$, the solution $V^{m,l,R}$ of 
$$
\left\{\begin{array}{l}
\ds dV^{m,l,R}_t= \Delta V^{m,l,R}_tdt+ \sum_{k\in \Z^d, \ |k-l|\leq R} A(x-k) dB^k_t \qquad {\rm in}\; \R^d\times (-m^2,+\infty)\\ 
\ds V^{m,l,R}_{-m^2}\equiv 0. 
\end{array}\right.
$$
For any $y\in Q_1(l)$, Lemma~\ref{lem.VVlR} gives  
\be\label{th1}
\E\left[|DV^m_s(y)-DV^{m,l,R}_s(y)|^2\right] \lesssim_{A,d} \left\{\begin{array}{ll}
  R^{-d} \ \ {\rm if } \ \ R^2/(s+m^2)\leq 1,\\[1.5mm] 
 \exp\{-R^2/(5(s+m^2))\} \ \  {\rm otherwise. } \end{array}\right.
\ee
We replace $DV^m_{-n^2}(y)$ by $DV^{m,k,R}_{-n^2}(y)$ in $Q_1(k)$.  In order to apply \eqref{th1}, we note that the  assumption on $n$, $m$ and $t$, the choice of $R$, and the facts that $m-n\geq 1$ and that $t+1\leq m$ imply that $m^2-n^2\geq m+n \geq (t+1+n^2)^{1/2}\geq R^2$. 
\smallskip

Hence, $R^2/(m^2-n^2)\leq1$ and, in view of \eqref{th1}, we have  
\be
 \E\left[|D(V_t^m-V_t^n)(x)|^2\right] \lesssim_{A,d} \Big( \sum_{k,k'\in \Z^d} A_{k,k'} + B_{k,k'}\Big), 
 \ee
 where 
%
%
$$
A_{k,k'}= \int_{Q_1(k)\times Q_1(k')} p(x-y,t+n^2)p(x-y',t+n^2) \E\left[DV^{m,k,R}_{-n^2}(y) \cdot DV^{m,k',R}_{-n^2}(y')\right]dydy'
$$
and 
\begin{align*}
B_{k,k'} = & \int_{Q_1(k)\times Q_1(k')} p(x-y,t+n^2)p(x-y',t+n^2) \\
 & \qquad R^{-d/2} \Big(\E^{1/2}\left[|DV^{m,k,R}_{-n^2}(y')|^2\right]+\E^{1/2}\left[|DV^{m}_{-n^2}(y)|^2\right]\Big)\ dydy'.
\end{align*} 
Using  \eqref{ineq.timecont} and \eqref{lkjnefjjn} we find  
\begin{align*}
\sum_{k,k' \in \Z^d} B_{k,k'}& \lesssim_{A,d}  R^{-d/2}  \sum_{k,k'\in \Z^d} \int_{Q_1(k)\times Q_1(k')} p(x-y,t+n^2)p(x-y',t+n^2)dydy' \underset{A}\sim R^{-d/2}.
\end{align*}
To estimate $A_{k,k'}$ note that, if  $|k-k'|\geq 2R+2$, then $DV^{m,k,R}$ and $DV^{m,k',R}$ are independent, and, hence,
\begin{align*}
A_{k,k'}&=  \int_{Q_1(k)\times Q_1(k')} p(x-y,t+n^2)p(x-y',t+n^2)
\E\left[DV^{m,k,R}_{-n^2}(y)\right] \  \E\left[ DV^{m,k',R}_{-n^2}(y')\right]dydy'.
\end{align*}

Recall that, since $V^m$ is stationary in space, we have 
\be\label{m111}\ds \E\int_{Q_1(k)} DV^{m}_{-n^2}(y)dy =\E  \int_{Q_1(k')} DV^{m}_{-n^2}(y)dy =0.\ee 

To make use of this property, we replace $A_{k,k'}$ by $A_{k,k'}'$ given by  
\begin{align*}
A_{k,k'}'&=  \int_{Q_1(k)\times Q_1(k')} p(x-y,t+n^2)p(x-y',t+n^2)
\E\left[DV^{m}_{-n^2}(y)\right] \ \E\left[ DV^{m}_{-n^2}(y')\right]dydy',
\end{align*}
and we note that, with an argument similar to the one above, we have 
$$
\sum_{|k-k'|\geq 2R+2} A_{k,k'} -\sum_{|k-k'|\geq 2R+2} A_{k,k'}' \lesssim_{A,d}R^{-d/2}. 
$$
Next we replace  $p(x-y,t+n^2)$ by $p(x-k,t+n^2)$ and $p(x-y,t+n^2)$ by $p(x-k',t+n^2)$ in $A_{k,k'}'$, noting that 
\begin{align*}
\max \{|p(x-y,t+n^2)-p(x-k,t+n^2)|, & \  |p(x-y',t+n^2)-p(x-k',t+n^2)|\}\\[1.5mm]
& \lesssim  (t+n^2)^{-1-d/2}\exp\{ -|x-k|^2/(4(t+n^2)\}.
\end{align*} 
%
Since  \eqref{m111} and  Lemma \ref{lem.boundDV} to control the remaining terms, we obtain 
\begin{align*}
A_{k,k'}'&\lesssim_{A,d}  C(t+n^2)^{-1-d}  \exp\{ -(|x-k|^2+|x-k'|^2)/(4(t+n^2)\}, 
\end{align*}
Summing the terms $A_{k,k'}'$ with $|k-k'|\geq 2R+2$, we find 
\begin{align*}
\sum_{|k-k'|\geq 2R+2} A_{k,k'}'
 \lesssim_{A,d} (t+n^2)^{-1-d}    \left(\sum_{k\in \Z^d} \exp\{ -|x-k|^2/(4(t+n^2))\}\right)^2\lesssim_{A,d} (t+n^2)^{-1}.\\
\end{align*}
On the other hand, if $|k-k'|\leq 2R+2$, then, \eqref{lkjnefjjn} yields 
\[
A_{k,k'} \lesssim_{A,d} \int_{Q_1(k)\times Q_1(k')} p(x-y,t+n^2)p(x-y',t+n^2),\\
\]
and, hence, 
\begin{align*}
& \sum_{|k-k'|\leq 2R+2} A_{k,k'} \lesssim_{A,d} \int_{|y-y'|\leq 2R+4} p(x-y,t+n^2)p(x-y',t+n^2)\\
& \lesssim_{A,d} (t+n^2)^{-d} \int_{|y-y'|\leq 2R+4} \exp\{- (|x-y|^2+ |x-y'|^2)/(2(t+n^2))\} \\
& \lesssim_{A,d} \frac{ R^d}{(t+n^2)^{d/2}}
\end{align*}
It follows that 
\begin{align*}
& \E\left[|D(V_t^m-V_t^n)(x)|^2\right]\lesssim_{A,d} \left(R^{-d/2} +(t+n^2)^{-1}+ \frac{ R^d}{(t+n^2)^{d/2}}\right),
\end{align*}
which completes the proof since $R= [(t+1+n^2)^{1/4}]$. 

\end{proof}

We have now all the ingredients needed to prove the main result.
\begin{proof}[Proof of Theorem \ref{thm:main}] In view of Lemma \ref{lem.Cauchy},  the sequence $(DV^n)_{n\in \N}$ converges along subsequences in $L^2(\Omega,L^2_{loc}(\R^d\times \R))$ to some $Z$, which is stationary in space, and solves 
\begin{equation}\label{eq:Zequation}
d Z_t(x)= \Delta Z_t(x)dt + \sum_{k\in \Z^d} DA(x-k) dB^k_t \ \ \text{in} \ \ \R^d\times \R,
\end{equation}
and, thus, is continuous in time and smooth in space.
\smallskip

Moreover, in view of the bound on $DV^n$ in Lemma~\ref{lem.boundDV}, for any $x\in\R^d$, we have 
%
$$
\sup_{t\in \R} \E\left[ |Z_t(x)|^2\right]+\sup_{t\in \R} \E\left[ |DZ_t(x)|^2\right]\lesssim_{A,d} 1.
$$

Fix  $t_0\in \R$,  let $V_{t_0}$ be the smooth antiderivative of $Z_{t_0}$ with, for definiteness,  $V_{t_0}(0)=0$, which exists since 
$Z$ is the limit of gradients, and  $V$  the solution of \eqref{eq:main} in $\R^d \times [t_0,+\infty)$ with initial condition $V_{t_0}$. It is immediate that $DV=Z$. 
\smallskip

Next we prove that $Z$ is the unique process  satisfying \eqref{eq:Zequation} which is stationary in space and satisfies the bounds $\sup_t \E[|Z_t|^2]<+\infty$.  
\smallskip

Let $Z'$ be another such process. Then, for any $i\in \{1, \dots, d\}$, $u(x,t)=(Z-Z')_i$ is an entire solution of the heat equation. It follows from a  classical estimate on the heat equation (see, for example,  \cite{EvBook}) that there exists $C>0$ such that, for any $r\in \R$, 
$$
\max_{(y,s)\in Q_{r/2}\times [t-r^2/4,t]} |Du(y,s)| \leq \frac{C}{r^{d+3}} \int_{t-r^2}^t\int_{Q_r(x)} |u(y,s)|dyds,
$$
and, hence, 
\begin{align*}
\max_{y\in Q_1(0)} |Du(y,t)|^2 & \leq \frac{C}{r^{d+4}} \int_{t-r^2}^t\int_{Q_r(x)} |u(y,s)|^2dyds .
\end{align*}
Taking expectation and  using  the stationarity of $u$ and the $L^2$ bound, we find 
\begin{align*}
\E\left[ \max_{y\in Q_1(0)} |Du(y,t)|^2\right]  & \leq  \frac{C}{r^2}. 
\end{align*}
This proves that $Z_i(\cdot,t)\equiv Z'_i(\cdot, t)$ for a fixed $t$, and, since $\E[\int_{Q_1}Z(x,t)dx]=\\
\E[\int_{Q_1}Z'(x,t)dx]=0$, it follows  that $Z\equiv Z'$.  
\smallskip

Finally,  we prove that $Z$ is stationary in time. For this, we note that the map $t\to Z_t$ is measurable with respect to the $\sigma-$algebra generated by $(B^i_{s\wedge t})_{s\leq t}$ because this is the case for the maps $t\to DV^n_t$. Therefore, there exists a measurable   ${\mathcal Z}$ such that $Z_t(x)={\mathcal Z}(t,x, (B^i_{\cdot \wedge t})_{i\in \Z^d})$.
\smallskip

Next we note that, for any $s\in \R$, $Z_{\cdot+s}(\cdot)$ solves the same equation as $Z_\cdot$ with Brownian motions  shifted  in time by $s$. Hence, by the uniqueness of the solution,  $Z_{t+s}(\cdot)= {\mathcal Z}(t,x, (B^i_{(\cdot +s)\wedge t})_{i\in \Z^d})$, which has the same law as $ {\mathcal Z}(t,x, (B^i_{\cdot \wedge t})_{i\in \Z^d})$. It follows  that the law of $Z_\cdot$ is the same as the law of $Z_{\cdot+s}$, thus, $Z$ is the stationary in time.  
\smallskip

The attractor property of $Z$, that is,   \eqref{attractor},  is a straightforward consequence of Lemma \ref{lem.Cauchy}. Indeed, choose $n=0$, $t>0$ and $m$ larger than $t+1$ in the lemma. Then 
$$
 \E\left[|D(V_t^m-V_t)(x)|^2\right] \lesssim_{A,d} t^{-(1\wedge (d/4))}.
$$
Letting $m\to+\infty$, the construction of $Z$, gives
$$
 \E\left[|Z_t(x)-DV_t(x)|^2\right] \lesssim_{A,d}t^{-(1\wedge (d/4))}.
$$
Integrating over $Q_1$ yields \eqref{attractor}. 

\end{proof}


\section{Random homogenization of uniformly elliptic nonlinear parabolic equation in divergence form}\label{sec.homogen}

In this section we investigate the random homogenization of 
\be\label{pde10}
\partial_t u^\ep -{\rm div}( a(Du^\ep,\frac{x}{\ep},\frac{t}{\ep^2}, \omega)) =f \ \ \text{in} \ \ \R^d\times (0,T) \quad 
u^\ep(\cdot,0)=u_0.
\ee
We start with the description of the environment in subsection~\ref{subsec.envi} and state the assumptions on the vector field in subsection~\ref{subsec.ass}.  Subsection~\ref{subsec.corrector} is about the existence of the corrector and introduces the effective vector field. The homogenization result is developed in subsection~\ref{subsec.homogen}.

\subsection{Description of the environment}\label{subsec.envi}
We fix an ergodic  environment probability, that is, assume that 
\be\label{omega}
\begin{cases}
\text{$(\Omega,\mathcal F, \P)$ is a probability  space
endowed with an ergodic semigroup }\\
\tau: \Z^{d}\times\R\times \Omega\to \Omega \ 
\text{ of measure preserving maps,}
\end{cases}
\ee
and  we denote  by ${\bf L^2}$ the set of stationary maps $u=u(x,t,\omega)$ meaning 
\be\label{m1}
u(x+k,t+s,\omega)= u(x,t, \tau_{(k,s)}\omega) \ \ \text{for all} \ \  (k,s,\omega)\in \Z^d\times \R\times \Omega,
\ee
and such that 
\be\label{m2}
\|u\|_{{\bf L^2}}= \E\left[ \int_{\tilde Q_1}u^2\right]<+\infty.
\ee
Note that, if $u \in {\bf L^2} $, the stationarity in time implies that the quantity $$\E\left[\int_{t_1}^{ t_2}\int_O u(x,s)dxds\right],$$ where $O$ is a bounded measurable subset of $\R^d$, is affine in $t_2-t_1$, and, therefore, the limit
$$
\E\left[\int_O u(x,t)dx\right] = \lim_{h\to 0^+}\E\left[ \frac{1}{2h}\int_O \int_{t-h}^{t+h} u(x,s)dx ds\right] 
$$
exists for any $t\in \R$ and is independent of $t$. 
\smallskip

Let ${\mathcal C}$ be the subset of ${\bf L^2}$  of maps with smooth and square integrable space and time  derivatives of all order belonging  to ${\bf L^2}$. A simple argument using mollification in $\R^{d+1}$ yields that ${\mathcal C}$ is dense in ${\bf L^2}$ with respect to the norm in \eqref{m2}.
\smallskip

We denote by ${\bf H^1}$ the closure of ${\mathcal C}$ with respect to  the norm 
$$
\|u\|_{{\bf H^1}}=(\|u\|_{{\bf L^2}}^2+ \|\partial_t u\|_{{\bf L^2}}^2+ \|Du\|_{{\bf L^2}}^2)^{1/2},
$$
while  ${\bf H^1_x}$ the closure of ${\mathcal C}$ with respect to  the norm 
$$
\|u\|_{{\bf H^1_x}}=(\|u\|_{{\bf L^2}}^2+\|Du\|_{{\bf L^2}}^2)^{1/2}
$$
and ${\bf H^{-1}_x}$ is its dual space. 
\smallskip

Moreover,  ${\bf L^2_{pot}}$ is  the closure with respect to  the  ${\bf L^2}$-norm of  $\{Du: \  u\in {\mathcal C}\}$ in $({\bf L^2}(\Omega))^d$.
\smallskip

For later use, we also note that,  in view of the stationarity, for all  $u, v\in {\bf H^1}$ and $i=1, \ldots, d$,
$$
\E\left[ \int_{\tilde Q_1} u\partial_{x_i} v\right] = - \E\left[ \int_{\tilde Q_1} v \partial_{x_i}u\right],
$$
and 
$$
\E\left[ \int_{\tilde Q_1} u\partial_t v\right] = - \E\left[ \int_{\tilde Q_1} v \partial_tu\right].
$$

Finally, given a nonnegative weight $\rho$, we write ${\bf L^2_{\rho}}$, ${\bf H^1_\rho}$ and ${\bf H^{-1}_\rho}$  for the  spaces in which the  norm is evaluated against $\rho$. 
\smallskip

Finally, we note that, whenever an equation is said to be solved in the sense of distributions, then the pairing is the standard and  not the weighted one. 


\subsection{The assumptions on the vector field}\label{subsec.ass}
We assume that the  vector field $a:\R^d\times\R^d\times \R \times \Omega \to \R^d$ is  
\be\label{T}
\begin{cases}
\text{space-time stationary,  and, }\\[1mm]  
\text{strongly monotone and Lipschitz continuous   uniformly in $x, t$ and $\omega$.}
\end{cases}
\ee
Moreover, it is assumed that 
\be\label{a2}
\text{ $|a(0,\cdot, \cdot,\cdot)|: \R^d \times \R \times \Omega \to \R \in {\bf L^2}$, }
\ee
and, hence, 
\begin{equation}\label{Hypa3}
\E\left[ \int_{\tilde Q_1} |a(0,x,t)|^2dxdt\right]< +\infty.
\end{equation}

%

\subsection{The existence of a corrector and the effective  nonlinearity}\label{subsec.corrector}
We prove here the existence, for each $p\in \R^d$,  of a corrector, that is a map $\chi^p:\R^d\times\R\times \Omega \to \R$ with $\partial_t \chi^p$ and $D\chi^p$ stationary and of mean $0$ and such that
\[\partial_t \chi^p -{\rm div}( a(p+D\chi^p,x,t,\omega))= 0 \ \ \text{in} \ \ \R^d\times \R,\]
and use it to define the effective vector field $\ol a:\R^d\to \R^d.$
\smallskip

The result is stated next.

\begin{thm}\label{thm.corrector} Assume \eqref{omega}, \eqref{T}
and \eqref{a2}. For any $p\in \R^d$, there exists a unique map $\chi^p:\R^{d+1}\times \Omega\to \R$ such that 
$$
\int_{\tilde Q_1}\chi^p(x,t,\omega)dxdt=0 \ \ \text{$\P-$a.s.},  \ \ D\chi^p\in {\bf L^2_{pot}}, 
 \ \  \partial_t\chi^p \in {\bf H^{-1}_x},
$$
and 
\be\label{eq.thetaw}
\partial_t \chi^p -{\rm div}( a(p+D\chi^p,x,t,\omega))= 0 \ \ {\rm in} \ \ {\bf H^{-1}_x}.
\ee
\vskip.075in
Moreover, as $\ep\to 0$ and  $\P-$a.s. and in expectation, 
$$
\chi^\ep(x,t;p,\omega)=\ep  \chi^p(\frac{x}{\ep},\frac{t}{\ep^2},\omega) \to 0 \ \  \text{ in} \ \  L^2_{\text{loc}}(\R^{d+1}).
$$
In addition, the vector field  $\ol a:\R^d \to \R^d$ defined by 
\begin{equation}\label{def.barabara}
\overline a(p)= \E\left[ \int_{\tilde Q_1} a( p+D\chi^p,y,\tau,\omega)dy d\tau\right]
\end{equation}
is monotone and Lipschitz continuous.

\end{thm}

The proof of Theorem~\ref{thm.corrector}  is long and technical. At first look, its structure appears to be similar to the ones of the analogous results for periodic and almost periodic media.  The standard approach  is  to consider the solution (approximate corrector) of a regularized version of the corrector equation with small second derivative in time to make the problem uniformly elliptic set  in a bounded domain and small discount factor to guarantee the solvability. The next step is to obtain uniformly apriori bounds for the space and time derivatives of the approximate corrector  and to pass to the weak limit, which yields an equation involving the weak limit of the time derivative and the divergence of the weak limit of the vector field.  Note that, due to the unboundedness of the domain it is necessary to use weighted space, a fact that introduces another layer of approximations and technicalities. 

\smallskip

The proof of Theorem~\ref{thm.corrector} is organized in a number of lemmata, which provide incremental information leading to the final argument. 

\smallskip

Throughout the proof, to justify repeated integration by parts and to deal with the unbounded domain, we use the exponential  exponential weight $\hat \rho_\theta$, which,  for $\theta>0$, is  given by 
$$
\hat \rho_\theta(x,t)=  \exp\{-\theta (1+|x|^2+t^2)^{1/2}\}. 
$$ 

The first lemma is about the existence of as well as some apriori bounds for the approximate corrector in a bounded domain. 

\begin{lem}\label{lem.2.2} Assume \eqref{omega}, \eqref{T}
and \eqref{a2}. For any $\omega\in \Omega$, $\lambda>0$ and $L>0$, let $u_L\in H^1_0(\tilde Q_L)$ be the solution of 
\be\label{m5}
\lambda u_L  -\lambda \partial_{tt}u_L+  \partial_t u_L-{\rm div}( a(Du_L+p, \omega)) =0 \ \ {\rm in } \ \ \tilde Q_L \quad u_L=0  \ \  {\rm in} \ \  \partial \tilde Q_L.
\ee
%
There exists $\theta_0>0$, which depends on $\lambda$ but not on $L$ or $\omega$, such that, for any $\theta\in (0,\theta_0]$ and $\P-$a.s., 
\begin{equation}\label{eq.bonbound}
\int_{\tilde Q_L} \left( \lambda u_L^2 + \lambda (\partial_t u_L)^2
+ |Du_L|^2\right)\hat \rho_\theta \lesssim_{p,\lambda, \theta} (1+\int_{\tilde Q_L} |a(0)|^2\hat \rho_\theta).
\end{equation}
\end{lem}

Note that the integral $\int_{\tilde Q_L} |a(0)|^2\hat \rho_\theta$ in the right-hand side of \eqref{eq.bonbound} is random and that the implicit constant does not depend on either $\omega$ or $L$

\begin{proof} Using  $\hat \rho_\theta u_L $ as a test function in \eqref{m5}, the  monotonicity and  Lipschitz continuity of $a $  and  the fact that $|D\hat \rho_\theta|+|\partial_t \hat \rho_\theta|\lesssim \theta \hat \rho_\theta$, we find 
\begin{align*}
& \int_{\tilde Q_L} \left( \lambda u_L^2 + \lambda (\partial_t u_L)^2
+ C_0^{-1}|Du_L+p|^2\right)  \hat \rho_\theta\\ 
& \; \leq -\int_{\tilde Q_L} \Bigl( \lambda \partial_t u_L u_L \frac{\partial_t \hat \rho_\theta}{\hat \rho_\theta} -\frac{(u_L)^2\partial_t \hat \rho_\theta}{2\hat \rho_\theta} - a(Du_L+p)\cdot p + u_L a(Du_L+p)\cdot \frac{D\hat \rho_\theta}{\hat \rho_\theta}\Bigr)\hat \rho_\theta \\ 
& \lesssim  \int_{\tilde Q_L}( \lambda \theta |\partial_t u_L||u_L|  + \theta u_L^2 +(|a(0)|+ |Du_L+p|)(|p| + \theta |u_L| )  \hat \rho_\theta,
\end{align*}
and, hence,  the claim. 

\end{proof}

Next,  we use Lemma~\ref{eq.bonbound} to obtain the  existence and bounds for approximate solutions of the approximate regularized problem in all of $\R^{d+1}$. 

\begin{lem}\label{lem.23} Assume \eqref{omega}, \eqref{T},
and \eqref{a2}.  For any $p\in \R^d$, $\lambda>0$ and $\theta\in (0,\theta_0)$, there exists, $\P-$a.s. and in the sense of distributions, a unique stationary  solution   $\chi^{\lambda,p}\in {\bf H^1_{\hat \rho_\theta}}$  of 
\begin{equation}\label{eq.approxcorr}
\lambda \chi^{\lambda,p}  -\lambda \partial_{tt}\chi^{\lambda,p}+  \partial_t \chi^{\lambda,p}-{\rm div}( a(D\chi^{\lambda,p}+p, \omega)) =0\ \  {\rm in }\ \ \R^{d+1},
\end{equation}
which is  independent of $\theta\in (0,\theta_0)$,  belongs to ${\bf H^1}$ and, in addition, 
\begin{equation}\label{inequlambda.esti1}
\E\left[ \int_{\tilde Q_1} \lambda (\chi^{\lambda,p})^2 + \lambda (\partial_t \chi^{\lambda,p})^2 + |D\chi^{\lambda,p}|^2\right] \lesssim_p 1 
\end{equation}
and, for all  $\phi\in {\bf H^1},$
\begin{equation}\label{inequlambda.esti2}
\E\left[\int_{\tilde Q_1} \partial_t \chi^{\lambda,p} \phi \right]\lesssim_p \lambda^{1/2} \|\phi\|_{{\bf H^1}}+ \|D\phi\|_{{\bf L^2}},
\end{equation}
both estimates being independent of  $\lambda$. 
\end{lem} 

\begin{proof} Let $u_L$ be as in Lemma \ref{lem.2.2}. The stationarity of $a$ and \eqref{Hypa3} yield
$$
\E\left[ \int_{\R^{d+1}} |a(0)|^2\hat \rho_\theta\right] \lesssim_{\theta, d} \E\left[ \int_{\tilde Q_1} |a(0)|^2\right] < +\infty,  
$$
Let $\omega\in \Omega$ be such that 
$$
\int_{\R^{d+1}} |a(0,x,t,\omega)|^2\hat \rho_\theta(x,t)dxdt <\oo$$
for a countable  sequence of $\theta\to 0$ and, thus, for any $\theta\in (0,1]$. Cleary, the set of such $\omega$ has  probability $1$. 
\smallskip

Fix such $\omega$.  It follows from Lemma \ref{lem.2.2}  that the family $(u_L)_{L\in (0,\oo)}$ is bounded in ${\bf H^1_{ \rho_{\theta}}}$ for any $\theta\in (0,\theta_0]$. A diagonal argument then yields a subsequence, which, to keep the notation simple, is denoted as the family,  and some $u\in \bigcap_{\theta'\in (0,\theta_0]} { H^1_{\hat \rho_{\theta'}}}$,
such that, as $L\to \infty$,  $u_L\rightharpoonup u$  in ${ H^1_{\hat \rho_{\theta}}}$ for any $\theta \in (0,\theta_0]$.
\smallskip

In particular,  $u_L\to u$ in $L^2(\tilde Q_R)$ for any $R>0$ and, therefore, in  $L^2_{\hat \rho_\theta}$ for all $\theta\in (0,\theta_0)$,  since, for any $R>0$, 
\begin{align*}
\|u_L-u\|_{L^2_{\rho_\theta}(\R^{d+1})} & \leq \|u_L-u\|_{L^2_{\rho_\theta}(\tilde Q_R)} +( \sup_{\R^{d+1}\backslash \tilde Q_R}\frac{\rho_\theta}{\rho_{\theta_0}})\|u_L-u\|_{L^2_{\rho_{\theta_0}(\R^{d+1}\backslash \tilde Q_R)}}\\[2mm]
& \leq \|u_L-u\|_{L^2(\tilde Q_R)} +\Big( \sup_{\R^{d+1}\backslash \tilde Q_R}\frac{\rho_\theta}{\rho_{\theta_0}})(\|u_L\|_{L^2_{\rho_{\theta_0}}(\R^{d+1})}+
\|u\|_{L^2_{\rho_{\theta_0}}(\R^{d+1})}\Big).
\end{align*} 
Note that above the first term in the right-hand side  tends to $0$ as $L\to \infty$ and the second one tends to $0$, uniformly in $L$, as $R\to +\infty$. 
\smallskip

We can also assume that, as $L\to \infty$,  $a(Du_L+p, \omega) \rightharpoonup \xi\in \bigcap_{\theta'\in (0,\theta_0]} { L^2}_{\hat \rho_{\theta'}}$. 
\smallskip

It follows that, in the sense of distributions, 
\begin{equation}\label{lekrgndf}
\lambda u  -\lambda \partial_{tt}u+  \partial_t u-{\rm div}( \xi) =0 \ \ {\rm in } \ \ \R^{d+1},
\ee
and, for all  $\theta\in (0,\theta_0]$,
\begin{equation}\label{lekrgndfBIS}
\int_{\R^{d+1}} \left( \lambda u^2 + \lambda (\partial_t u)^2
+ |Du|^2\right)  \hat \rho_\theta\lesssim_{p, \lambda, \theta)}(1+ \int_{\R^{d+1}} |a(0)|^2\hat \rho_\theta).
\end{equation}

Next we  check that $u$ is a solution of \eqref{eq.approxcorr}. In what follows, we use that $u\in  \bigcap_{\theta'\in (0,\theta_0]} { H^1_{\hat \rho_{\theta'}}}.$
\smallskip

 
Let $\phi\in C^\infty_c(\R^{d+1})$. The strong monotonicity of $a$ gives, for $L$ large enough, 
\begin{align*}
\int_{\tilde Q_L}\Bigl( \lambda (u_L-\phi)^2+ \lambda (\partial_t u_L-\partial_t \phi)^2 + 
(a(Du_L+p)-a(D\phi+p))\cdot D(u_L-\phi) \Bigr) \hat \rho_\theta  \geq 0. 
\end{align*}
Moreover,  using  $u_L\hat \rho_\theta$ as a test function for the equation of $u_L$,  we find 
\begin{align*}
& \int_{\tilde Q_L} \Bigl( \lambda u_L^2 + \lambda (\partial_t u_L)^2 +  \lambda  \partial_t u_L u_L  \frac{\partial_t \hat \rho_\theta}{\hat \rho_\theta}+ \partial_t u_Lu_L 
+ a(Du_L+p)\cdot Du_L 
\\
& \qquad\qquad  
+ u_L  a(Du_L+p) \cdot \frac{D\hat \rho_\theta}{\hat \rho_\theta}  \Bigr)\hat \rho_\theta =0.
\end{align*}
Hence, 
\begin{align*}
&\int_{\tilde Q_L}\Bigl( -2 \lambda u_L\phi+ \lambda \phi^2 -2 \lambda \partial_t u_L\partial_t \phi+\lambda (\partial_t \phi)^2 - a(Du_L+p)\cdot D\phi \\
& -a(D\phi+p)\cdot D(u_L-\phi) 
  - \partial_t u_L u_L
- \lambda \partial_t u_L u_L\frac{\partial_t \hat \rho_\theta}{\hat \rho_\theta}\
- u_L a(Du_L+p) \cdot \frac{D\hat \rho_\theta}{\hat \rho_\theta} 
 \Bigr) \hat \rho_\theta  \geq 0. 
\end{align*}
 
Letting $L\to +\infty$  and recalling that, as $L\to \infty$,  $u_L\to u$, $\partial_t u_L \rightharpoonup \partial_t u$, $Du_L\rightharpoonup Du$ and $a(Du_L+p)\rightharpoonup \xi$ in ${ L^2_{\hat \rho_\theta}}$ and, hence, in $L^2_{loc}$,  we  obtain 
\begin{align*}
&\int_{\R^{d+1} }\Bigl( -2 \lambda u\phi+ \lambda \phi^2 -2 \lambda \partial_t u\partial_t \phi+\lambda (\partial_t \phi)^2 - \xi\cdot D\phi -a(D\phi+p)\cdot D(u-\phi) \\
& \qquad\qquad -\partial_t u u 
- \lambda \partial_t u u\frac{\partial_t \hat \rho_\theta}{\hat \rho_\theta}
- u \xi \cdot \frac{D\hat \rho_\theta}{\hat \rho_\theta} 
 \Bigr) \hat \rho_\theta  \geq 0. 
\end{align*}
On the other hand, integrating \eqref{lekrgndf} against $\phi\hat \rho_\theta$, we get  
\begin{align*}
\int_{\R^{d+1}} \Bigl( \lambda u  \phi +\lambda \partial_{t}u\partial_t \phi + \lambda \partial_{t}u \phi\frac{\partial_t\hat \rho_\theta}{\hat \rho_\theta}
+  \partial_t u \phi +\xi \cdot D\phi + \phi\xi\cdot \frac{D\hat \rho_\theta}{\hat \rho_\theta} \Bigr)\hat \rho_\theta  =0.
\end{align*}
Inserting the last equality into the inequality above gives
\begin{align*}
&\int_{\R^{d+1} }\Bigl( - \lambda \phi(u-\phi) - \lambda \partial_t \phi(\partial_t u-\partial_t \phi) -a(D\phi+p)\cdot D(u-\phi) \\
& \qquad\qquad -\partial_t u (u-\phi) 
- \lambda \partial_t u (u-\phi) \frac{\partial_t \hat \rho_\theta}{\hat \rho_\theta}
- (u-\phi) \xi \cdot \frac{D\hat \rho_\theta}{\hat \rho_\theta} 
 \Bigr) \hat \rho_\theta  \geq 0. 
\end{align*}
Using $\phi= u+h\psi$ for $h>0$ small and $\psi\in C^\infty_c(\R^{d+1})$, something that can be done using standard approximation arguments, yields, after dividing  by $h$ and letting $h\to 0$,  
\begin{align*}
&\int_{\R^{d+1} }\Bigl(  \lambda u\psi + \lambda \partial_t u \partial_t \psi + a(Du+p)\cdot D\psi 
+\partial_t u \psi 
+ \lambda \partial_t u \psi \frac{\partial_t \hat \rho_\theta}{\hat \rho_\theta}
+ \psi \xi \cdot \frac{D\hat \rho_\theta}{\hat \rho_\theta} 
 \Bigr) \hat \rho_\theta  \geq 0. 
\end{align*}
The facts that $\psi$ has a compact support, $u$ and its derivatives are locally integrable and, as $\theta\to0$,  the derivatives of $\hat \rho_\theta$ tend to $0$ locally uniformly,  gives, after letting  $\theta\to0$,
\begin{align*}
&\int_{\R^{d+1} } \lambda u\psi + \lambda \partial_t u \partial_t \psi + a(Du+p)\cdot D\psi 
+\partial_t u \psi  \geq 0. 
\end{align*}
Since $\psi\in C^\infty_c(\R^{d+1})$ is arbitrary, the last inequality implies  that $u$ is a solution of \eqref{eq.approxcorr} in the sense of distributions. 
\smallskip

Next we check that $u$  is unique among weak solutions of \eqref{eq.approxcorr} in ${ H^1_{\hat \rho_\theta}}$ for some $\theta>0$. 
\smallskip

Let $u_1, u_2$ be two solutions and set $\tilde u= u_1-u_2$. Using $\tilde u \hat \rho_\theta$ as a test function in the equation for $\tilde u$, we find 
\begin{align*}
& \int_{\R^{d+1}} (\lambda \tilde u^2 + \lambda (\partial_t \tilde u)^2  +(a(p+Du_1)-a(p+Du_2))\cdot D\tilde u ) \hat \rho_\theta  \\
& \qquad = - \int_{\R^{d+1}} \lambda \tilde u \partial_t \tilde u\partial_t  \hat \rho_\theta+ \tilde u (a(p+Du_1)-a(p+Du_2))\cdot D  \hat \rho_\theta \\ 
& \lesssim_{a} \theta \int_{\R^{d+1}} (\lambda |\tilde u| |\partial_t \tilde u|+ C_0 |D\tilde u | |\tilde u| ) \hat \rho_\theta. 
\end{align*}
Then a standard argument based on Cauchy-Schwartz inequality implies that, for $\theta$ small enough, $\tilde u\equiv0$.  
\smallskip

Since  \eqref{eq.approxcorr} has a  unique solution in ${ H^1_{\hat \rho_\theta}}$ for some $\theta>0$, the whole family  $u_L$ converges to $u$ as $L\to +\infty$.  It follows that 
$u$ is measurable in $\Omega$.  Moreover, the  stationarity of the equation and the  uniqueness of the solution imply that $u$ is also  stationary. 
\smallskip

To establish the bounds claimed, we test   the equation for $u$ against $\hat \rho_\theta u$.  Using the monotonicity  of $a$ and arguing as above we get 
\begin{align*}
& \int_{\R^{d+1}} \left( \lambda u^2 + \lambda (\partial_t u)^2
+ C_0^{-1}|Du+p|^2\right)  \hat \rho_\theta\\ 
&\lesssim_a \int_{\R^{d+1}}( \lambda \theta |\partial_t u||u|  + \theta u^2 +(|a(0)|+ |Du+p|)(|p| + \theta |u| )  \hat \rho_\theta,
\end{align*}
 It follows that, for $\theta$ small enough depending on $p$ but independent of $\omega$,  
\begin{align*}
& \int_{\R^{d+1}} (\lambda  u^2 + \lambda (\partial_t  u)^2  +C_0^{-1}|Du+p|^2 ) \hat \rho_\theta  \lesssim_p \int_{\R^{d+1}} |a(0)|^2 \hat \rho_\theta. 
\end{align*}
Taking expectations and using \eqref{a2} and the fact $u$ is stationary and \eqref{a2} gives gives  \eqref{inequlambda.esti1}.  
\vskip.075in
Finally, to obtain  \eqref{inequlambda.esti2} we  use the equation and \eqref{inequlambda.esti1}. 

\end{proof}

In order to proceed, we need the following  remark about  the reconstruction of a map from its derivatives. 
\begin{lem}\label{lem.defu} Assume \eqref{omega} and let $\theta\in {\bf H^{-1}_x}$ and $w\in {\bf L^2_{pot}}$ satisfy, for all $\phi\in \mathcal C$ and $i=1,\dots,d$,   the compatibility condition  
\be\label{compcond}
\lg \theta, \partial_{x_i}\phi\rg_{{\bf H_{x}^{-1}},{\bf H^1_x}} = \E\left[ w_i \partial_t \phi\right] 
\ee
Then there exists a  measurable map $u:\R^{d+1}\times \Omega\to \R$ such that, a.s., $\int_{\tilde Q_1}u(x,t,\omega)dxdt=0$, $Du= w$ and $\partial_t u= \theta$ in the sense of distributions. 
\end{lem}

For the proof, we need to use regularizations (convolutions) with  a kernel $K^\ep(x,t)= \ep^{-(d+1)}K(x/\ep,t/\ep)$ for $K:\R^{d+1}\to [0,+\infty)$  smooth, nonnegative, symmetric, compactly  supported and such that $\int_{\R^{d=1}} K dxdt=1$. 
\smallskip

For $u\in {\bf L^2}$,  define
$$
K^\ep\ast u(x,t,\omega)= \int_{\R^{d+1}} K^\ep(x-y,t-s) u(y,s,\omega)dyds.
$$
It is a classical fact that $K^\ep\ast u$ belongs to ${\mathcal C}$ and that
$$
\lim_{\ep\to0} \|u-K^\ep\ast u\|_{{\bf L^2}}=0. 
$$

\begin{proof}[The proof of Lemma~\ref{compcond}]  Fix $\ep>0$ and define $\theta^\ep \in {\bf H^{-1}_x}$ and $w^\ep$ so that, for all $\phi\in {\bf H^1_x}$ and $w^\ep = K^\ep \ast w$, 
\be\label{def.thetaep}
\lg \theta^\ep, \phi\rg_{{\bf H^{-1}_x},{\bf H^1_x}}= \lg \theta, K^\ep\ast \phi\rg_{{\bf H^{-1}_x},{\bf H^1_x}}
\ee

It is immediate that  $\theta^\ep, w^\ep$ belong to ${\mathcal C}$ and, in view of \eqref{compcond}, for all $i=1,\dots,d$, 
$$
\partial_{x_i} \theta^\ep= \partial_t w^\ep_i. 
$$
It follows  that there exists a measurable and smooth in $x,t$ map $u^\ep:\R^{d+1}\times \Omega\to \R$ such that $\partial_t u^\ep= \theta^\ep$,  $Du^\ep= w^\ep$, and, without loss of generality,  $\int_{\tilde Q_1}u^\ep= 0$.  
\smallskip

For any $R\geq 1$, Poincar\'e's inequality gives  (see, for instance, the proof of Lemma 4.2.1 in \cite{Kry}) 

$$
\|u^\ep(\cdot, \cdot,\omega)\|_{{ L^2}(\tilde Q_R)} \lesssim_{d,R} \|Du^\ep(\cdot, \cdot, \omega)\|_{{ L^2}(\tilde Q_R)}+ \|\partial_t u^\ep (\cdot, \cdot,\omega)\|_{{ L^2}(I_R, { H^{-1}}(Q_R))},
$$
and, thus, 
\begin{align*}
\E\left[ \|u^\ep(\cdot, \cdot,\omega)\|_{{ L^2}(\tilde Q_R)}^2\right] \lesssim_{d,R} \E\left[\|w^\ep\|_{{ L^2}}^2+ \|\theta^\ep\|_{{ L^2}(I_R, { H^{-1}}(Q_R))}^2 \right].
\end{align*}
Using a  diagonal argument, we can find  $\ep_n\to 0$  and  $u\in L^2_{loc}(\R^{d+1}\times \Omega)$ such that, for any $R$,  $u^{\ep_n} \rightharpoonup u$ in ${ L^2}(\tilde Q_R\times \Omega)$. 
\smallskip

It is, then,  easy to check  that $Du=w$, $\partial_t u= \theta$ and $\int_{\tilde Q_1}u=0$.

\end{proof}

We use Lemma~\ref{lem.defu} to obtain the following result which is one of the most crucial steps for the construction of the corrector.  
\begin{lem}\label{lem.Ewxi=0} Assume \eqref{omega}. If    $\theta\in {\bf H^{-1}_x}$, $w\in {\bf L^2_{pot}}$ and $\xi\in {\bf L^2}$ satisfy the compatibility condition  \eqref{compcond} and 
$$
\theta -{\rm div}( \xi)= 0 \ \ {\rm in} \ \ {\bf H^{-1}_x},
$$
then 
\[ \E\left[ \int_{\tilde Q_1} w\cdot \xi\right]=0.\]
\end{lem} 

\begin{proof} 
Let $\theta^\ep$ be defined by \eqref{def.thetaep}, $w^\ep= K^\ep\ast w$ and $\xi^\ep= K^\ep \ast \xi$. Then, 
$$
\theta^\ep -{\rm div}(\xi^\ep)=0.
$$
Lemma \ref{lem.defu} and its proof yield a measurable in $\omega$ and smooth in $(x,t)$ map $u^\ep:\R^{d+1}\times \Omega\to \R$
such that $Du^\ep= w^\ep$ and $\partial_t u^\ep= \theta^\ep$, and, in the classical sense,  
\be\label{kjhnrsdtf}
\partial_t u^\ep -{\rm div} (\xi^\ep)= 0 \ \  {\rm in} \ \ \R^{d+1}\times \Omega.
\ee
Arguing by contradiction,  we  assume  that $$\hat \kappa=\E\left[\int_{\tilde Q_1} \xi \cdot w \right] >0.$$
Since the map
$
t\to \E\left[ \int_{Q_1} w(x,t)\cdot \xi(x,t)dx\right]
$
is well-defined and constant,  we actually have, for all $t\in \R$, 
\be\label{kljznesrdOLD}
\E\left[ \int_{Q_1} w(x,t)\cdot \xi(x,t)dx\right]=\hat \kappa>0. 
\ee
In view of the stationarity of  $w$ and $\xi$,  \eqref{kljznesrdOLD} implies that there exist $\ep_0>0$ and $0<\kappa<\hat \kappa$ such that, for all $t\in \R$,  $\ep\in (0,\ep_0)$ and $R>0$, 
\be\label{kljznesrd}
\E\left[ \int_{Q_R} w^\ep(x,t)\cdot \xi^\ep (x,t)dx\right]\geq  \kappa R^d. 
\ee

Fix $R>0$ and let $\psi=\psi_R \in C^1(\R^d\times [0,+\infty))$ be such that  
$$\begin{cases}\psi(x,R)= 0 \ \ \text{ in} \ \  \R^d\backslash Q_{R+1}, \ \  \psi(x,R)= 1 \ \ \text{ in} \ \  Q_R,  \ \ \|D\psi\|_\infty+\|\partial_R\psi\|_\infty\lesssim_d 1,\\[1.5mm]
 \text{ and}  \ \  |D\psi(x,R)| \lesssim_d  \partial_R \psi(x,R).\end{cases}$$
Note  that such $\psi$  can  be constructed  by convolving  in space the map  ${x\to \bf 1}_{Q_{R+1/2}}(x)$ with 
a nonnegative kernel with  sufficiently small support. 
\smallskip

Finally, for some $c_0>0$ and $T$ sufficiently large to be chosen later,  set $R(t)= (T-c_0 t)^{1/2}$. 
\smallskip

Then 
\begin{align*}
& \frac{d}{dt} \int_{\R^d}\frac{(u^\ep)^2(x,t)}{2} \psi(x,R(t)) dx \notag\\
& \qquad =  R'(t)\int_{\R^d}  \frac{(u^\ep)^2(x,t)}{2} \partial_R \psi(x,R(t)) dx + \int_{\R^d}u^\ep(x,t)\partial_t u^\ep(x,t)\psi(x,R(t))  dx \notag\\
& \qquad   =R'(t)\int_{\R^d}  \frac{(u^\ep)^2(x,t)}{2} \partial_R \psi(x,R(t)) dx - \int_{\R^d}\xi^\ep(x,t)\cdot w^\ep(x,t)\psi(x,R(t))dx \notag\\
& \qquad \qquad - \int_{\R^d}u^\ep(x,t) w^\ep(x,t)\cdot  D\psi(x,R(t))dx.
\end{align*}

Young's inequality yields, for any $\alpha>0$, 
\begin{align*}
& \frac{d}{dt} \int_{\R^d}\frac{(u^\ep)^2(x,t)}{2} \psi(x,R(t)) dx \notag \\
& \qquad   \leq  R'(t)\int_{\R^d}  \frac{(u^\ep)^2(x,t)}{2} \partial_R \psi(x,R(t)) dx - \int_{Q_{R(t)}}\xi^\ep(x,t)\cdot w^\ep(x,t)dx \notag \\
& \qquad + \int_{Q_{R(t)+1}\backslash Q_{R(t)}}|\xi^\ep(x,t)|\ |w^\ep(x,t)|dx +\alpha |R'(t)| \int_{\R^d}\frac{(u^\ep)^2(x,t)}{2} |D\psi(x,R(t))| dx \notag
\\
& \qquad  
+ \dfrac{C}{\alpha|R'(t)|} \int_{Q_{R(t)+1}\backslash Q_{R(t)}} |w^\ep(x,t)|^2 dx.
\end{align*}
Recall that, by construction,   $R'<0$, $\|D\psi\|_\infty\lesssim_d C$ and $|D\psi|\lesssim_d \partial_R \psi$
\smallskip

Hence, choosing from now on $\alpha$ small enough depending only on  $d$, taking expectations and using \eqref{kljznesrd}, we find
\begin{align}\label{kqhjsndkjnjj}
 \frac{d}{dt}\E\left[ \int_{\R^d}\frac{(u^\ep)^2(x,t)}{2} \psi(x,R(t)) dx\right] 
&  + \kappa (R(t))^d  - \E\left[ \int_{Q_{R(t)+1}\backslash Q_{R(t)}}|\xi^\ep(x,t)|\ |w^\ep(x,t)|dx\right] \notag \\
&   
\lesssim_\alpha \dfrac{1}{|R'(t)|}\E\left[ \int_{Q_{R(t)+1}\backslash Q_{R(t)}} |w^\ep(x,t)|^2 dx\right].
\end{align}
We use next  the stationarity of $w^\ep$ and $\xi^\ep$, and the facts that $|Q_{R(t)+1}\backslash Q_{R(t)}| \lesssim C(R(t))^{d-1}$,    $\xi^\ep, w^\ep \in {\bf L^2}$, and  $R'(t)= c_0(R(t))^{-1}$ to get, for some $C>0$, 

\begin{align*}
 \frac{d}{dt} \E\left[\int_{\R^d}\frac{(u^\ep)^2(x,t)}{2} \psi(x,R(t)) dx\right] &  \leq - (R(t))^d \kappa + C(|R'(t)|^{-1}+1)(R(t))^{d-1}\\
& = -(R(t))^d (\kappa -Cc_0^{-1}-C(R(t))^{-1}),
\end{align*}

Choosing $c_0>1$ large so that $\kappa -Cc_0^{-1}\geq \kappa/2$ and $t\leq t_T=T-16C^2\kappa^{-2}c_0^{-1}$, in order to have   $C(R(t))^{-1}\leq \kappa/4$ on $[0,t_T]$, we find, for all $t\in [0,t_T]$,  

\begin{align*}
 \frac{d}{dt} \E\left[\int_{\R^d}\frac{(u^\ep)^2(x,t)}{2} \psi(x,R(t)) dx\right] &  \leq - (R(t))^d \dfrac{\kappa}{4}. 
\end{align*}

Integration in time  over  $t\in [h,c_0^{-1}T]$ for $h\in [0,T^{1/2}]$ (note that,  if  $c_0$ and $T$ are large enough, $c_0^{-1}T <t_T$) and the fact that $\psi\geq0$ give
$$
\E\left[\int_{\R^d}\frac{(u^\ep)^2(x,h)}{2} \psi(x,R(h)) dx\right]   \ge \dfrac{\kappa}{4} \int_h^{c_0^{-1}T}(R(t))^d dt.
$$
 Integrating once more in time over  $h\in [0,T^{1/2}]$ and noting  that, since 
$R(h)\le T^{1/2}$, $ \psi(x,R(h))\le 1_{Q_{T^{1/2}+1}},$ we get 
\begin{align*}
& \E\left[\int_0^{T^{1/2}}\int_{Q_{T^{1/2}+1}}\frac{(u^\ep)^2(x,h)}{2}  dxdh\right] 
\geq C^{-1} \kappa T^{(d+3)/2}. 
\end{align*}
Our goal  is to apply \ref{lem.appendix} in the Appendix. For this,  we note that, since $Du^\ep= w^\ep \in {\bf L^2_{pot}}$,   $\E[\int_{Q_1}Du^\ep(\cdot,t)]=0$. Moreover,  in view of \eqref{kjhnrsdtf} and the fact that  $\xi^\ep\in {\bf L^2}$ is stationary,
$$
\E\Big[\int_{\tilde Q_1} \partial_t u^\ep\Big]=\lg \partial_t u^\ep, 1\rg_{{\bf H^{-1}_x, H^1_x}}= \lg {\rm div}(\xi^\ep), 1\rg_{{\bf H^{-1}_x, H^1_x}}=0
$$ 

Hence, we can apply Lemma~\ref{lem.appendix} which implies that, for any $\delta>0$,  there exists $R_\delta$ such that, for all   $R\geq R_\delta$, 
$$
 \E\left[\int_0^R\int_{Q_R} (u^\ep(x,h))^2dxdh\right]\leq \delta R^{d+3}. 
$$ 
Choosing  $R= T^{1/2}+1$ and $T$ large, we obtain 
\begin{align*}
 \frac{\delta}{2} (T^{1/2}+1)^{d+3} & \ge  \E\left[\int_0^{T^{1/2}}\!\!\!\int_{Q_{T^{1/2}+1}}\!\!\!\!\!\frac{(u^\ep)^2(x,h)}{2}  dxdh\right] 
 \geq C^{-1} \kappa T^{\frac{d+3}{2}},
\end{align*}
which yields a contradiction if  $\delta$ is small enough and $T$ is  large enough. 
\smallskip

It follows that we must have  $$\E\left[\int_{\tilde Q_1} \xi\cdot w\right] \leq 0.$$ 
 
 Arguing similarly  for negative $t$ gives  the opposite inequality. 

\end{proof}

The next lemma is the step that provides the sought after corrector as well as the properties (monotonicity and Lipschitz continuity) of $\ol a$.
\begin{lem} \label{m10} Assume \eqref{omega}, \eqref{T},
and \eqref{a2}. For any $p\in \R^d$ there exists a unique pair $(\theta^p, w^p)\in {\bf H^{-1}_x}\times {\bf L^2_{pot}}$ satisfying  
\eqref{compcond} and 
\be\label{eqthetaw}
\theta^p-{\rm div}(a(w^p+p,x,t,\omega))=0 \ \  {\rm in} \ \   {\bf H^{-1}_x}. 
\ee
Moreover, for all  $p,p'\in \R^d$, 
\be\label{eq.reguwp}
\|w^p-w^{p'}\|_{{\bf L^2}}\lesssim_a  |p-p'|. 
\ee
Finally, the vector field $\ol a$ defined by \eqref{def.barabara} is  monotone and  Lipschitz continuous.  
\end{lem}

\begin{proof} Let $\chi^{\lambda, p}$ be given by  Lemma \ref{lem.23}. In view of 
\eqref{inequlambda.esti1} and \eqref{inequlambda.esti2}, there exist  a subsequence $\lambda_n\to 0$, $w\in {\bf L^2_{pot}}$, $\theta \in {\bf H^{-1}}$ and  $\xi \in {\bf L^2}$ such that $D\chi^{\lambda_n,p}
\rightharpoonup  w$,  $\partial_t \chi^{\lambda_n, p} \rightharpoonup \theta$ , and $a(D\chi^{\lambda_n,p}+p) \rightharpoonup  \xi$ in their respective spaces. 
\smallskip

Moreover, in view of \eqref{inequlambda.esti2}, for all $\phi\in {\bf H^1}$, 
$$
\lg \theta, \phi\rg_{{\bf L^2}} \lesssim_p \|D\phi\|_{{\bf L^2}}
$$
which means that, in fact,  $\theta \in {\bf H^{-1}_x}$. 
\smallskip

Note also that, since the pair $(\partial_t \chi^{\lambda,p}, D \chi^{\lambda,p})$ satisfies  \eqref{compcond}, so does $(\theta, w)$.
\smallskip

Finally, \eqref{eq.approxcorr} implies 
\be\label{oizuaesnf}
\theta - {\rm div}(\xi)= 0  \ \  {\rm in} \ \  {\bf H^{-1}_x}. 
\ee

It remains to check that \eqref{eqthetaw} holds. As we show below, this is a consequence of the monotonicity of $a$, which gives that,  for any test function $\phi\in \mathcal C$,  
$$
\E\left[\int_{\tilde Q_1} \lambda (\chi^{\lambda,p}-\phi)^2+ \lambda (\partial_t \chi^{\lambda,p}-\partial_t \phi)^2 + (a(D\chi^{\lambda,p}+p)-a(D\phi+p))\cdot (D\chi^{\lambda,p}-D\phi) \right] \geq 0.
$$
Multiplying \eqref{eq.approxcorr} by $\chi^{\lambda,p}$ and taking expectation, we find  
$$
\E\left[\int_{\tilde Q_1} \lambda (\chi^{\lambda,p})^2+ \lambda (\partial_t \chi^{\lambda,p})^2 + a(D\chi^{\lambda,p}+p)\cdot D\chi^{\lambda,p})\right]= 0,
$$
and, thus,  
\begin{align*}
\E\Bigl[\int_{\tilde Q_1} \lambda (-2 \chi^{\lambda,p}\phi+\phi^2)+ \lambda (-2\partial_t \chi^{\lambda,p}\partial_t \phi+(\partial_t \phi)^2) - a(D\chi^{\lambda,p}+p)\cdot D\phi \\
  -a(D\phi+p)\cdot (D\chi^{\lambda,p}-D\phi)))\Bigr] \geq 0.
\end{align*}
Passing  to the limit $\lambda_n \to 0$, in view of the estimates on $\chi^{\lambda,p}$ we get 
\begin{align*}
\E\left[ \int_{\tilde Q_1} - \xi \cdot D\phi -a(D\phi+p)\cdot (w-D\phi) \right]\geq 0.
\end{align*}
Since this last inequality   holds for any $\phi\in {\mathcal C}$, we also have, for any $z\in {\bf L^2_{pot}}$, 
$$
\E\left[ \int_{\tilde Q_1}- \xi \cdot z -a(z+p)\cdot (w-z) \right]\geq 0.
$$
Choose $z= w+\theta z'$ with  $z'\in {\bf L^2_{pot}}$. Then, after dividing by $\theta$ and letting $\theta\to0$, in view of Lemma \eqref{lem.Ewxi=0},  we get 
$$
\E\left[ \int_{\tilde Q_1}- \xi \cdot z' +a(w+p)\cdot z' \right]\geq \limsup_{\theta\to0} \frac{1}{\theta} \E\left[\int_{\tilde Q_1} \xi\cdot w\right]= 0,
$$
Since  the last inequality holds for any $z'\in {\bf L^2_{pot}}$, we infer that 
\be\label{m9}
\E\left[ \int_{\tilde Q_1}- \xi \cdot z' +a(w+p)\cdot z' \right] =0.
\ee
Going  back to \eqref{oizuaesnf}, \eqref{m9}  implies that, for any $\phi\in {\bf H^1_x}$, 
$$
\lg \theta -{\rm div}(a(w+p)), \phi\rg_{{\bf H^{-1}_x},{\bf H^1_x}} =   \E\left[\int_{\tilde Q_1}  \xi \cdot D\phi -a(w+p)\cdot D\phi \right] =0,
$$
and, hence,   $(\theta, w)$ satisfies \eqref{eqthetaw}. 
\smallskip

Next we prove at the same time the uniqueness of $(\theta, w)$ and the  monotonicity  of $\overline a$. 
\smallskip

Let $p^1\in \R^d$ and $(\theta^1,w^1)$ be a solution associated with $p^1$, and  set $\xi^1= a(w^1+p^1)$. Then 
$$
\theta-\theta^1 -{\rm div}(\xi-\xi^1)=0.
$$
Applying Lemma \ref{lem.Ewxi=0}  to the pair $(\theta-\theta^1,w-w^1),$ we find  $$\E\left[ \int_{\tilde Q_1}(\xi-\xi^1)\cdot (w-w^1)\right]=0.$$

The monotonicity of $\ol a$ follows from the following calculation that uses the fact that, since  $w-w^1\in {\bf L^2_{pot}}$, we have $\E[\int_{\tilde Q_1} w-w^1]=0$:
\begin{align*}
&( \overline a(p)-\overline a(p^1))\cdot (p-p^1) = \E\left[ \int_{\tilde Q_1}(a(w+p)- a(w^1+p^1))\cdot (w+p-w^1-p^1)\right]\\
&\qquad  \geq C_0^{-1} \E\left[ \int_{\tilde Q_1}|w+p-w^1-p^1|^2\right] = C_0^{-1} (\E\left[ \int_{\tilde Q_1}|w-w^1|^2\right]+|p-p^1|^2).
\end{align*}

The uniqueness of $(\theta, w)$ also follows from the inequality above.  Indeed set  $p^1=p$. It follows that   $w=w^1$, which  in turn implies that  $\theta=\theta^1$. 
\smallskip

The Lipschitz continuity follows from the observation that 
\begin{align*}
|\overline a(p)-\overline a(p^1)| & \leq \E\left[ \int_{\tilde Q_1} | a(w+p)- a(w^1+p^1)| \right] \leq C_0(\E^{1/2}\left[\int_{\tilde Q_1} |w-w^1|^2\right]+ |p-p^1| )\\
& \leq C_0( \left(( \overline a(p)-\overline a(p^1))\cdot (p-p^1)\right)^{1/2}+ |p-p^1|) \leq \frac12 |\overline a(p)-\overline a(p^1)|+ C|p-p^1|.
\end{align*}
Note that the above  also yields \eqref{eq.reguwp}.

\end{proof}

We have now all the necessary ingredients to prove Theorem~\ref{thm.corrector}.
\begin{proof}[Proof of Theorem \ref{thm.corrector}] Fix $p\in \R^d$ and  let $(\theta^p,w^p)$ and $\chi^p$ 
be given respectively by Lemma~\ref{m10} and Lemma \ref{lem.defu}.  
\smallskip

Then, for  $\chi^\ep(x,t;p,\omega)=\ep  \chi^p(\dfrac{x}{\ep},\dfrac{t}{\ep^2},\omega)$ and $a^\ep(p,x,t,\omega)= a(p, \dfrac{x}{\ep}, \dfrac{t}{\ep^2},\omega)$, we have 
\be\label{eq.xiep}
\partial_t \chi^\ep-{\rm div}(a^\ep(p+D\chi^\ep, x,t))=0\ \ {\rm in} \ \ R^d\times \R.
\ee
%
%
First we show that there exists a universal constant $C_0$ such that, $\P-$a.s. and for any $R,T>0$, 
\begin{align}\label{Step11}
&   \underset{\ep \to 0} \limsup \int_{0}^{T} \int_{Q_R} (\chi^\ep(x,t))^2dxdt  \leq C_0T^{3}R^{d-2}\E\left[ \int_{\tilde Q_1} |a(D\chi+p)|^2\right].  
\end{align}
Fix  $\xi \in C^\oo(\R;[0,1]) $ such that $\xi\equiv 0$ in  $(-\infty,-1)$, $\xi\equiv 1$ in $[0,+\infty)$ and $\xi'\leq 2$ and set 
$$
\phi(x,s,t)=\xi\big((3-2t^{-1}s) -R^{-1}|x|_\infty\big). 
$$
We note for later use that, since  $1\leq (3-2t^{-1}s)\leq 3$ for $s\in[0,t]$, $\phi(x,s,t)=1$ in $Q_R$, while $\phi(x,s,t)=0$ in $\R^d\backslash Q_{4R}$. 
\smallskip

Using the equation satisfied by $\chi^\ep$  and Young's inequality,  we find, for any $t>0$ fixed and any $s\in (0,t)$, 
\begin{align*}
& \frac{d}{ds} \int_{\R^d} \frac12 (\chi^\ep(s))^2\phi(x,s,t) =
\int_{\R^d}  \frac12 (\chi^\ep)^2 \partial_s \phi
- (a^\ep D\chi^\ep \phi + \chi^\ep a^\ep D\phi) .
\\[1.5mm]
&                                                                   
\leq \int_{\R^d}  \frac12 (\chi^\ep)^2 \partial_s \phi
- a^\ep D\chi^\ep \phi  + 
\frac{1}{2}R^{-1}t |a^\ep|^2 |D\phi|+ \frac{1}{2} Rt^{-1}(\chi^\ep)^2 |D\phi|.
\end{align*}

The computation above, which  here is made at a formal level, can be  easily be rigorous by regularizing $\chi^\ep$ by convolution. 
\smallskip

Since  $\partial_s\phi= -2t^{-1}\xi' $ while $|D\phi|\leq R^{-1}\xi'$, we can  absorb the last term in the righthand side into the first one to obtain
\begin{align*}
& \frac{d}{ds} \int_{\R^d} \frac12 (\chi^\ep(s))^2\phi(x,s,t) \leq 
 \int_{\R^d}  
- a^\ep D\chi^\ep \phi + R^{-2}t  \int_{Q_{4R}} |a^\ep|^2  .
\end{align*}
Integrating  the above inequality  in time, between $0$ and $t$ and using  the definition of $\phi$ we get  
\begin{align*}
& \int_{Q_R} \frac12 (\chi^\ep(t))^2dx \leq \int_{Q_{4R}} \frac12 (\chi^\ep(0))^2dx\\
& \qquad  -
\int_{0}^{t}  \int_{\R^d}  
a^\ep(s) D\chi^\ep(s) \phi(x,s,t) dxds+ R^{-2}t  \int_{0}^{t} \int_{Q_{4R}} |a^\ep(s)|^2 dxds.
\end{align*}
A second integration in $t\in (0, T)$ gives 
\begin{align*}
&  \int_{0}^{T} \int_{Q_R} \frac12 (\chi^\ep(t))^2 dxdt \leq  T \int_{Q_{4R}} \frac12 (\chi^\ep(0))^2dx\\
& \qquad  -
\int_{0}^{T} \int_{0}^{t}  \int_{\R^d}  
 a^\ep(s) D\chi^\ep(s) \phi(x,s,t)  dx dsdt \\
 & \qquad  + R^{-2} \int_{0}^{T}t\int_{0}^{t} \int_{Q_{4R}} |a^\ep(s)|^2  dxdsdt.
\end{align*}
We now let $\ep\to 0$. It follows from 
 Lemma \ref{lem.appendix}
and the ergodic theorem that,  $\P-$a.s., 
\begin{align}\label{kqeusrndxc}
&\underset{\ep \to 0}  \limsup \int_{0}^{T} \int_{Q_R} \frac12(\chi^\ep(t))^2 dxdt \leq \\
& \qquad -
\int_{0}^{T} \int_{0}^{t} \int_{\R^d}  \E\left[ \int_{\tilde Q_1} a(D\chi+p)\cdot D\chi\right] \phi(x,s,t) dx dsdt \notag \\
& \qquad  +R^{-2}T \int_{0}^{T} \int_{0}^{t} \int_{Q_{4R}} 
\E\left[ \int_{\tilde Q_1} |a(D\chi+p)|^2\right]dxdsdt .
\end{align}
Lemma \ref{lem.Ewxi=0} gives that  the first term in the right-hand side vanishes. Thus,  
\begin{align*}
&\underset{\ep \to 0} \limsup   \int_{0}^{T} \int_{Q_R} \frac12(\chi^\ep(t))^2dxdt \lesssim R^{d-2}T^3\E\left[ \int_{\tilde Q_1} |a(D\chi+p)|^2\right],
\end{align*}
and, hence,   \eqref{Step11}.
\smallskip

A symmetric argument yields that, $\P-$a.s.,   
\begin{align}\label{lem.convconv}
&   \underset{\ep \to 0}\limsup \int_{-T}^{T} \int_{Q_R} (\chi^\ep(x,t))^2dxdt  \leq C_0T^{3}R^{d-2}\E\left[ \int_{\tilde Q_1} |a(D\chi+p)|^2\right].  
\end{align}

Next we show the  convergence of $(\chi^\ep)$ to $0$. 
\smallskip

Let $\omega\in \Omega$ be such that  \eqref{lem.convconv} holds for any $T,R>0$. Then, in view of \eqref{lem.convconv},  the families $(\chi^\ep)_{\ep>0}$,  $(D\chi^{\ep})_{\ep>0}$ and $(\partial_t \chi^{\ep})_{\ep>0}$ are respectively  is bounded in $L^2_{loc}(\R^d\times \R)$, $L^2_{loc}(\R^d\times \R)$ and  $L^2_{loc}(H^{-1})$. 
Hence, in view of the classical Lions-Aubin Lemma~\cite{Au, Li},  the  family  $(\chi^{\ep})_{\ep>0}$ is relatively compact in $L^2_{loc}(\R^{d+1})$. 
\smallskip

Let  $(\chi^{\ep_{n}})$ be any converging subsequence with limit   $\chi$ in $L^2_{loc}(\R^d\times \R)$. Since  $a$ and $D\chi^\ep$ are stationary in an ergodic environment, $a^\ep(D\chi^\ep+p)$ converges weakly to a constant. Thus, in view of \eqref{eq.xiep},   
$\chi$ solves $\partial_t   \chi= 0 \ \ \text{in} \ \ \R^d\times \R$. 
Dividing \eqref{lem.convconv} and letting $T\to 0$ yields that 
$ \chi(\cdot,0)=0.$  
\smallskip

Therefore $ \chi\equiv 0$,  and, hence,  $\chi^{\ep_{n}}\to 0$ in $L^2_{loc}(\R^d\times \R)$.

\end{proof}

\subsection{Homogenization}\label{subsec.homogen}

We now turn to the homogenization of \eqref{pde10}. 
The  aim is to show that the family   
$(u^\ep)_{\ep>0}$ converges  to the solution $u$ of the homogenous equation 
\be\label{eq.AHlim}
\partial_t \ol u -{\rm div}(\overline a(D\ol u)) =f(x,t) \ \  {\rm in} \ \  \R^d\times (0,T) \quad 
\ol u(\cdot,0)=u_0  \ \  {\rm in} \ \ \R^d,
\ee
where $\bar a:\R^d\to \R$ is defined by \eqref{def.barabara}, see below for a precise statement.

For the statement and the proof of the result we will use again the weight 
\begin{equation}\label{def.rhothetaBIS}
\rho_\theta(x):= \exp\{ -\theta (1+|x|^2)^{1/2}\}
\end{equation}
and we will work in the weighted spaces  $L^2_{\rho_\theta}=L^2_{\rho_\theta}(\R^d)$, $H^1_{\rho_\theta}=H^1_{\rho_\theta}(\R^d)$, etc... 
\smallskip

The homogenization result is stated next. 

\begin{thm}\label{thm.homo} Assume \eqref{omega}, \eqref{T}, 
and \eqref{a2} and let $\ol a:\R^d\to\R^d$  be the monotone and Lipschitz continuous vector field defined by  \eqref{def.barabara}. Then, 
 for every $T>0$, $u_0\in L^2(\R^d)$ and $f \in L^2(R^d\times (0,T))$, if $u^\ep$ and $\ol u$ solve respectively \eqref{pde10} and  \eqref{eq.AHlim}, then, 
 $\P-$a.s. and in expectation,  $u^\ep(\cdot,t) \to \ol u(\cdot,t)$  in $L^2_{\rho_\theta}(\R^d\times (0,T))$ for any $\theta>0$. 
\end{thm} 

The argument is long. To help the reader we split it in several parts (subsubsections).  In the first  subsubsection we prove a refined energy estimate for solutions of \eqref{pde10}. 
Then, in subsubsection~\ref{def.Omega0E} we identify  $\Omega_0\subset \Omega$  of full measure where the homogenization takes place. In subsubsection~\ref{subsubseq} we extract a subsequence $\ep_n\to 0$ along which $u^{\ep_n}$ has a limit.  To show that this limit satisfies the effective PDE, we construct a special  test function in subsubsection~\ref{subsubsectionTestF}. Theorem~\ref{thm.homo} is proved in subsubsection~\ref{thm.homo1}. The last three subsubsections are devoted to the proof of some technical parts used in subsubsection~\ref{thm.homo1}.

\subsubsection{Preliminary estimates}\label{subsubsec.preli}

A solution to \eqref{pde10} is a  measurable map $u^\ep:\R^d\times [0,T]\times \Omega\to \R$ such that, $\P-$a.s., $u^\ep(\cdot, \cdot, \omega) \in L^2([0,T], H^1_{\rho_\theta})\cap C^0([0,T], L^2_{\rho_\theta})$ which   satisfies the equation in the sense of distributions. Since, $\P-$a.s.,   $a(0,\cdot, \cdot,\omega) \in L^2_{loc}(\R^d\times (0,T])$,  $u^\ep(\cdot, \cdot, \omega)$ exists and is unique. 
\smallskip

In the next lemma we sharpen the standard energy estimate for solutions of \eqref{pde10}. 

\begin{lem}\label{lem.estiuep} Assume \eqref{omega}, \eqref{T},
and  \eqref{a2},  $u_0\in L^2(\R^d)$ and $f \in L^2(R^d\times (0,T))$. There exists  $C_\theta^\ep(\omega)>0$, which is $\P-$a.s.  finite,  converges, as  $\ep\to 0$, in   $L^1(\Omega)$, and   depends  on $\theta$, $T$, $\|f\|_2$ and the  monotonicity and Lipschitz constants of $a$ such that 
\begin{equation}\label{eq.Comega}
\sup_{t\in [0,T]} \|u^\ep(\cdot, t)\|_{L^2_{\rho_\theta}}^2+ \int_0^T\|Du^\ep(\cdot, t)\|_{L^2_{\rho_\theta}}^2 dt +
\int_0^T \|\partial_t u^\ep\|_{H^{-1}_{\rho_\theta}}^2 \leq C^\ep_\theta(\omega).
\end{equation}
\end{lem}

\begin{proof}
%
Throughout the proof, to simplify the notation, in place  of   $a(Du^\ep, \dfrac{x}{\ep}, \dfrac{t}{\ep^2},\omega)$, we write $a^\ep(Du^\ep)$. 
\smallskip

It is immediate that,  for a.e. $t\in (0,T]$, $u^\ep$ satisfies  the standard energy inequality
\begin{align*}
&\int_{\R^d} \frac12 u^\ep(t)^2\rho_\theta -\int_{\R^d} \frac12 u_0^2\rho_\theta = \int_0^t \int_{\R^d} - a^\ep(Du^\ep) \cdot (Du^\ep \rho_\theta+
u^\ep D\rho_\theta) + fu^\ep \rho_\theta \\
& \qquad \leq \int_0^t \int_{\R^d}  (-C_0^{-1} |Du^\ep|^2  + |a^\ep(0)| |Du^\ep|+ \theta |u^\ep| (|a^\ep(0)|+ C_0|Du^\ep| )+|f| |u^\ep|) \rho_\theta \\
& \qquad \leq \int_0^t \int_{\R^d} (-\frac{1}{2C_0} |Du^\ep|^2  +  C(|a^\ep(0)|^2+  |u^\ep|^2+ |f|^2)\rho_\theta\\
& \qquad \leq \int_0^t \int_{\R^d} (-\frac{1}{2C_0} |Du^\ep|^2  +  C|u^\ep|^2)\rho_\theta + C_\theta(\tilde C_\theta^\ep(\omega)+1), 
 \end{align*}
 where $C_\theta$ is a constant which depends only on $\theta$, $T$, $\|f\|_2$ and $C_0$ in \eqref{T}  (and might change from line to line) and 
 $$
\tilde C_\theta^\ep(\omega)= \int_0^T \int_{\R^d}|a^\ep(0,x,t)|^2 \rho_\theta(x)dx. 
 $$  
 It then follows from Gronwall's Lemma that
 $$
 \sup_{t\in [0,T]} \|u^\ep(\cdot, t)\|_{L^2_{\rho_\theta}}^2+ \int_0^T\|Du^\ep(\cdot, t)\|_{L^2_{\rho_\theta}}^2 dt \leq C_\theta(1+\tilde C^\ep_\theta(\omega)).
$$
 To estimate 
 $\partial_t u^\ep$,  we use  $\phi \rho_\theta$ with  $\phi\in C^\infty_c(\R^d\times [0,T])$ as a test function in \eqref{pde10}  and get  
\begin{align*}
 \int_0^T \lg \partial_t u^\ep, \phi \rg_{H^{-1}_{\rho_\theta}(\R^d), H^1_{\rho_\theta}(\R^d)} 
& = \int_0^T \int_{\R^d} -a^\ep(Du^\ep)\cdot D\phi \rho_\theta -a^\ep(Du^\ep)\cdot D\rho_\theta  \phi +f \phi\rho_\theta \\[1.5mm]
& \leq C_\theta ( \|a^\ep(0)\|_{L^2_{\rho_\theta}} + C_0\|u^\ep\|_{L^2(H^1_{\rho_\theta})} +\|f\|_{L^2_{\rho_\theta}}) \|\phi\|_{L^2(H^1_{\rho_\theta})},
\end{align*}
 and, in view of the previous estimate on $u^\ep$, 
\begin{align*}
 \int_0^T \|\partial_t u^\ep\|_{H^{-1}_{\rho_\theta}}^2 dt \leq C_\theta ( \tilde C_\theta^\ep +1).
\end{align*}
To complete the proof, we note that the ergodic Theorem implies that  $\tilde C_\theta^\ep$ converges, $\P-$a.s. and in $L^1(\Omega)$, to 
 $$
 \E\left[\int_0^T\int_{\R^d} |a(0,x,t)|^2 \rho_\theta(x)
dxdt\right]<+\infty.
$$
\end{proof}

\subsubsection{The identification of $\Omega_0$.} \label{def.Omega0E}
 Let $\chi^\ep(x,t;p,\omega)=\ep \chi (\frac{x}{\ep}, \frac{t}{\ep^2};p,\omega)$, where $\chi (y, \tau ;p,\omega)$ is the corrector found in 
 Theorem \ref{thm.corrector}. We know from Theorem \ref{thm.corrector} that $\chi^\ep$ solves in the sense of distributions the corrector equation 
$$
\partial_t \chi^\ep -{\rm div}(a(p+D\chi^\ep, \dfrac{x}{\ep}, \dfrac{t}{\ep^2}, \omega))=0\ \ \text{in} \ \ \R^d\times \R,
$$
and satisfies 
\be\label{cvchiep}
\lim_{\ep \to 0} \int_{\tilde Q_R}|\chi^\ep|^2 = 0\ \  \P-{\rm a.s.}.
\ee
In addition, since,  for each $p\in \R^d$,   $a(p+ D\chi, \cdot,\cdot,\cdot) \in {\bf L^2}$ and stationary,  the ergodic theorem yields,  for any cube $\tilde Q$ and any $g\in L^2(\tilde Q, \R^d)$ and  $\P-$a.s., 
\begin{equation}\label{eq.cvDchiep}
\int_{\tilde Q} g(x,t) \cdot a(p+ D\chi^\ep(x,t; p),  \dfrac{x}{\ep}, \dfrac{t}{\ep^2}, \omega)dxdt \underset{\ep\to 0}\to \int_{\tilde  Q} g(x,t)\cdot \bar a(p)dxdt.
\end{equation}
Similarly,  in view of the stationarity of $D\chi$,  for any $g\in L^2(\tilde Q)$ and $\P-$a.s., 
\begin{equation}\label{eq.cvDchiepBIS}
\int_{\tilde Q} g(x,t) |D\chi^\ep(x,t; p)|^2dxdt \underset{\ep\to 0}\to \E\left[  \int_{\tilde  Q} g(x,t) |D\chi(x,t;p)|^2\right].
\end{equation} 
Finally,  Lemma \ref{lem.Ewxi=0} yields
\begin{equation}\label{eq.cvDchiepTER}
\int_{\tilde Q} g(x,t) a^\ep(p+D\chi^\ep(x,t),x,t)\cdot D\chi^\ep(x,t; p) dxdt \underset{\ep\to 0}\to 0. 
\end{equation} 

Hence, given a countable family ${\mathcal E}$ dense in $\R^d$ and the (countable) family ${\mathcal Q}$ of  cubes with rational coordinates, we can find using a diagonal argument a set $\Omega_1$ of full probability  such that, for any $\omega\in \Omega_1$,  any $p\in {\mathcal E}$ and  $\tilde D\in {\mathcal Q}$, \eqref{cvchiep}, \eqref{eq.cvDchiep}, \eqref{eq.cvDchiepBIS} and \eqref{eq.cvDchiepTER} hold. 
\smallskip

Let  $\Omega_2$ be the full measure subset of $\Omega$ such that, for any $\omega\in \Omega_2$, the limit of the constant $C_\theta^\ep(\omega)$ in \eqref{eq.Comega} exists and is finite for any (rational) $\theta>0$ and such that, for $\ep_0=\ep_0(\omega)>0$ small enough and every $R>0$, 
\begin{equation} \label{eq.choixomega3}
\sup_{\ep\in (0,\ep_0)} \int_{\tilde Q_R} ( |a^\ep(0,x,t)|^2 + |D\chi^\ep(x,t;p)|^2 )dxdt <+\infty. 
\end{equation} 

The full measure subset of $\Omega$ in which homogenization takes place is $\Omega_0=\Omega_1 \cap \Omega_2$. Heretofore,  
we always work with $\omega\in \Omega_0$.

\subsubsection{Extracting a subsequence}\label{subsubseq} Fix $\omega \in \Omega_0$. 
In view of  \eqref{eq.Comega}, we know that the family  $(u^\ep)_{\ep>0}$ is compact in $L^2_{loc}(\R^{d+1})$. 
\smallskip

Let $(u^{\ep_{n}})_{n\in \N}$ be a converging sequence  with limit $u$. Then, for any $\theta>0$,  
\be\label{alkejzred}
\begin{cases}
 u^{\ep_{n}} \to u \; {\rm in }\; L^2_{\rho_\theta}(\R^d\times (0,T)),\ \  Du^{\ep_{n}} \rightharpoonup Du\;  {\rm in }\; L^2_{\rho_\theta}(\R^d\times (0,T)), \ \ \text{and}\\[1.5mm]  
\qquad \qquad \qquad  a^{\ep_{n}}(Du^{\ep_{n}}, \dfrac{x}{\ep_{n}}, \dfrac{t}{\ep^2_{n}},\omega) \rightharpoonup \xi\; 
{\rm in }\; L^2_{\rho_\theta}(\R^d\times (0,T)). 
\end{cases}
\ee
The aim is to prove that $u$ is the unique solution to \eqref{eq.AHlim}, which will then yield the a.s. convergence of $u^\ep$ to $u$. 
\smallskip

Heretofore, we work along this  particular subsequence $\ep_{n}$, which we denote by $\ep$ to simplify the notation. Note that, in view of \eqref{pde10}, we have 
\be\label{laerkjnsrdc}
\partial_t u -{\rm div} (\xi) = f(x,t) \ \  {\rm in} \ \ \R^d\times (0,T) \quad u(\cdot,0)=u_0 \ \  {\rm in} \ \  \R^d.
\ee
In addition, in view of   \eqref{eq.Comega},  for any $\theta>0$, we have 
\begin{equation}\label{eq.Comega2BIS}
\sup_{t\in [0,T]} \|u(\cdot, t)\|_{L^2_{\rho_\theta}}^2+ \int_0^T\|Du(\cdot, t)\|_{L^2_{\rho_\theta}}^2 dt +
\int_0^T \|\partial_t u\|_{H^{-1}_{\rho_\theta}}^2 \leq C_\theta(\omega),
\end{equation}
where $C_\theta(\omega)= \sup_{\ep\in (0,\ep_0]} C_\theta^\ep(\omega)$ is finite,  for $\ep_0$ small, since by the  construction of $\Omega_0$, $C_\theta^\ep(\omega)$ has a limit  as $\ep\to 0$. Similarly to  the construction of the corrector, we need to prove that we can replace $\xi$ by $\bar a(Du)$ in \eqref{laerkjnsrdc}.  

\subsubsection{The test functions}\label{subsubsectionTestF} Following the usual approach to prove homogenization for divergence form elliptic equations, given a test function $\phi\in C^\infty_c(\R^d\times [0,T))$, we need to consider, for each $\ep>0$,  the corrector $\chi^\ep(x,t,\omega)=\chi(\dfrac{x}{\ep}, \dfrac{t}{\ep^2}, D\phi (x,t),\omega)$ and work with $D\chi^\ep$ . The dependence on $D\phi$ creates technical problems since we do not have enough information about the regularity of the map  $p\to \chi(\cdot,\cdot,p,\omega)$.
\smallskip

To circumvent this difficulty, we  introduce a localization argument for  the gradient of the corrector, which is  based on a piecewise constant approximation of  $D\phi$.
\smallskip

Fix $\delta\in(0,1)$ and consider a locally finite family $(\hat Q_k )_{k\in \N}$  of disjoint cubes $\hat Q_k=Q_{R_k}(x_k)\times (t_k-T_k,t_k+T_k)$ in $\mathcal Q$ with  $T_k+R_k\leq \delta$ 
covering  $\R^d\times [0,T]$ up to a set of $0$ Lebesgue measure.  
\smallskip 

Let 
$$
p_k = \dashint_{\tilde Q_k} D\phi(x,t)dxdt,
$$
and, for each $k$, choose $p_k^\delta$ in  the countable family ${\mathcal E}$ defined in subsection \ref{def.Omega0E} and is such that $|p_k-p^\delta_k|\leq \delta$.
\smallskip

The localizations of $D\phi$ and $D\chi^\ep$ are
\be\label{m40}
D\tilde \phi (x,t) =\sum_k p_k^\delta {\bf 1}_{\hat Q_k} \ \  {and} \ \ D\tilde \chi^\ep(x,t,\omega)= \sum_{k} D\chi^\ep  (x,t; p_k^\delta).
\ee

Note that above we abused notation, since neither $D\tilde \phi$ nor $D\tilde \chi^\ep$ are gradients. We use, however,  the gradient symbol in order to stress the fact that they are respectively close  to $D\phi$ and $D\chi^\ep$. Indeed, we note, for later use, that $D\tilde \phi$ and  $D\tilde \chi^\ep$depend on $\delta$ and that $D\tilde \phi$ converges, as $\delta \to 0^+$, uniformly to $D\phi$.
 
%
\smallskip

Finally, we fix a smooth nonincreasing function $\zeta:[0,T]\to \R$ such that   $\zeta(0)=1$ and  $\zeta(1)=0$. 

\subsubsection{The proof of Theorem \ref{thm.homo}}\label{thm.homo1}  We prove that  $\xi=\ol a(D\ol u)$ in \eqref{laerkjnsrdc}. 
\smallskip

We write for simplicity below $a^\ep(p,x,t)$ for $a(p,\dfrac{x}{\ep},\dfrac{t}{\ep^2}, \omega)$.
\smallskip

The monotonicity of $a$ gives 
\begin{align*}
 \int_0^T  \int_{\R^d} \Big(a^\ep\big(Du^\ep(x,t),x,t\big)-a^\ep\big (D\tilde  \phi(x,t)+D\tilde \chi^\ep(x,t), x,t\big )\Big)\qquad\qquad\\
 \qquad \qquad \cdot (Du^\ep(x,t)-D\tilde  \phi(x,t)-D\tilde \chi^\ep(x,t))\rho_\theta(x)\zeta(t) dxdt \geq 0.
\end{align*}
Multiplying \eqref{pde10}  by $u^\ep\rho_\theta \zeta$ and integrating in space and time we find 
\begin{align*}
&- \int_{\R^d} \frac{u_0^2(x)}{2}  \rho_\theta(x)dx  - \int_0^T \int_{\R^d}  \frac{(u^\ep(x,t))^2}{2}  \rho_\theta(x)\zeta' (t)dxdt\\
& \qquad  +\int_0^T\int_{\R^d} a^\ep(Du^\ep(x,t),x,t)\cdot Du^\ep(x,t) \rho_\theta(x)\zeta(t)dxdt  \\
& \qquad  +\int_0^T\int_{\R^d} u^\ep a^\ep(Du^\ep(x,t),x,t)\cdot D \rho_\theta\zeta dx dt 
= \int_0^T\int_{\R^d}  f(x,t) u^\ep(x,t)\rho_\theta(x)\zeta(t).
\end{align*}

Subtracting the last two expressions we obtain  
\begin{align}\label{eq.lhjeznrgdgf}
& \int_{\R^d} \frac{u_0^2}{2}  \rho_\theta   + \int_0^T \int_{\R^d}  \frac{(u^\ep)^2}{2}  \rho_\theta\zeta' + \int_0^T \int_{\R^d} \Bigl( 
 -a^\ep(Du^\ep) \cdot (D\tilde \phi+ D\tilde \chi^\ep) .\notag \\   
\\
 &  -a^\ep (D\tilde \phi+D\tilde \chi^\ep)\cdot (Du^\ep-D\tilde  \phi-D\tilde \chi^\ep)+ f u^\ep\Bigr) \rho_\theta\zeta 
 -  \int_0^T \int_{\R^d} u^\ep a^\ep (Du^\ep) \cdot D \rho_\theta \zeta  \geq 0.\notag
\end{align}
To let $\ep \to 0$ in the above inequality, we 
first note that, in view of  \eqref{alkejzred}, 
\begin{align*}
& \lim_{\ep\to 0} \int_{\R^d} \frac{u_0^2}{2}  \rho_\theta  + \int_0^T \int_{\R^d}  \frac{(u^\ep)^2}{2}  \rho_\theta\zeta' + \int_0^T \int_{\R^d} \Bigl( 
 -a^\ep(Du^\ep) \cdot D\tilde \phi+ f u^\ep\Bigr) \rho_\theta\zeta  \\
 & \qquad \qquad   -  \int_0^T \int_{\R^d} u^\ep a^\ep (Du^\ep) \cdot D \rho_\theta \zeta   \\
 &\qquad = \int_{\R^d} \frac{u_0^2}{2}  \rho_\theta  + \int_0^T \int_{\R^d}  \frac{u^2}{2}  \rho_\theta\zeta' + \int_0^T \int_{\R^d} \Bigl( 
 -\xi \cdot D\tilde \phi + f u\Bigr) \rho_\theta\zeta \\
  & \qquad \qquad   -  \int_0^T \int_{\R^d} u \xi  \cdot D \rho_\theta \zeta .
\end{align*}
We claim that 
\begin{align}
& \lim_{\ep\to 0}  \int_0^T\int_{\R^d} a^\ep (D\tilde \phi(x,t)+D\tilde \chi^\ep(x,t), x,t) \cdot (D\tilde \phi(x,t)+D\tilde\chi^\ep(x,t)) \rho_\theta(x)\zeta(t)  dxdt \notag\\
& \qquad = 
\int_0^T\int_{\R^d} \bar a(D\tilde \phi(x,t))\cdot  D\tilde \phi(x,t) \rho_\theta(x)\zeta(t) dxdt, \label{limlim0}
\end{align}
and 
\begin{align}\label{limlim1}
&\lim_{\ep\to 0} \int_0^T\int_{\R^d} a^\ep(Du^\ep(x,t),x,t)\cdot D\tilde \chi^\ep(x,t)\rho_\theta(x)\zeta(t)dxdt \\
& \qquad + \int_0^T\int_{\R^d} a^\ep(D\tilde \phi(x,t)+D\tilde \chi^\ep(x,t), x,t) \cdot Du^\ep(x,t)\rho_\theta(x)\zeta(t)  dxdt \notag\\
& \qquad = \int_0^T\int_{\R^d} \overline a(D\tilde \phi(x,t)) \cdot Du(x,t) \rho_\theta(x)\zeta(t) dxdt. \notag
\end{align}
Assuming  \eqref{limlim0} and \eqref{limlim1}, we proceed with the ongoing proof. Passing to the $\ep\to 0$ limit in \eqref{eq.lhjeznrgdgf}, we find
\begin{align*}
& \int_{\R^d} \frac{u_0^2}{2}  \rho_\theta  + \int_0^T \int_{\R^d}  \frac{u^2}{2}  \rho_\theta\zeta' + \int_0^T \int_{\R^d} \Bigl( 
 -\xi  \cdot D\tilde \phi  -\bar a(D\tilde \phi) \cdot (Du-D\tilde \phi) + f u\Bigr) \rho_\theta\zeta  \\
& \qquad  -  \int_0^T \int_{\R^d} u \xi  \cdot D \rho_\theta \zeta  \geq 0.
\end{align*}
Next we  let $\delta\to 0$. Since, $D\tilde \phi\to D\phi$ uniformly,  we obtain,
\begin{align}
& \int_{\R^d} \frac{u_0^2}{2}  \rho_\theta + \int_0^T \int_{\R^d}  \frac{u^2}{2}  \rho_\theta\zeta' + \int_0^T \int_{\R^d} \Bigl( 
 -\xi  \cdot D \phi  -\bar a(D\phi) \cdot (Du-D \phi) + f u\Bigr) \rho_\theta\zeta  \notag \\
& \qquad  -  \int_0^T \int_{\R^d} u \xi  \cdot D \rho_\theta \zeta  \geq 0.
\label{aoiuscn}
\end{align}
while  using $\phi \rho_\theta\zeta $ as a test function in  \eqref{laerkjnsrdc} yields
\begin{align*}
&0= 
- \int_{\R^d} u_0\phi(0)\rho_\theta +
\int_0^T \int_{\R^d} u(-\partial_t \phi \rho_\theta \zeta -\phi\rho_\theta \zeta') +\xi\cdot (D\phi\rho_\theta\zeta + \phi D\rho_\theta\zeta)  -f   \phi \rho_\theta\zeta . 
\end{align*}
Combining the equation above and  \eqref{aoiuscn} we get 
\begin{align}
& \int_{\R^d} (\frac{u_0^2}{2}-u_0\phi(0)) \rho_\theta  + \int_0^T \int_{\R^d}  u(\frac{u}{2}-\phi)  \rho_\theta\zeta' \notag\\
& \;  + \int_0^T \int_{\R^d} \Bigl( 
  -\bar a (D\phi) \cdot (Du-D \phi) + f (u-\phi) - u \partial_t \phi \Bigr) \rho_\theta\zeta \label{izuakzesdnfc} \\
& -  \int_0^T \int_{\R^d} (u-\phi) \xi  \cdot D \rho_\theta \zeta  \geq 0.  \notag
\end{align}
We choose $\phi= u^\sigma +s\psi$ where $s>0$, $\psi\in C^\infty_c(\R^d\times [0,T))$ and $u^\sigma$ is a smooth approximation of $u$ with compact support in $\R^d\times [0,T]$ such that, as $\sigma \to 0$, 
\[
 u^\sigma(\cdot,0)\to u_0, \ \ 
u^\sigma \to u \ \ \text{ and} \ \  Du^\sigma \to  Du  \ \ \text{in} \ \ L^2_{\rho_\theta}, \ \ \text{and} \ \ \partial_t u^\sigma \to \partial_t u \ \ \text{ in} \ \  L^2(H^{-1}_{\rho_\theta});\]
note that such an approximation is possible in view of \eqref{eq.Comega2BIS}. 
\smallskip

We prove below that 
\begin{align}\label{limlim3}
& \lim_{\sigma\to 0} \int_{\R^d} (\frac{u_0^2}{2}-u_0u^\sigma(0))\rho_\theta  + \int_0^T \int_{\R^d} u(\frac{u}{2}-u^\sigma) \rho_\theta\zeta'  - \int_0^T \int_{\R^d}  u \partial_t u^\sigma \rho_\theta\zeta =0.
\end{align}

Thus,  in the limit  $\sigma\to 0$, \eqref{izuakzesdnfc} becomes 
\begin{align*}
& -s \int_{\R^d}u_0 \psi(0) \rho_\theta -s \int_0^T \int_{\R^d}  u\psi  \rho_\theta\zeta' +
s\int_0^T \int_{\R^d} \Bigl( \bar a (Du+sD\psi) \cdot D\psi -  f \psi  -  u \partial_t \psi \Bigr) \rho_\theta\zeta \\
& \qquad  +s \int_0^T \int_{\R^d} \psi \xi  \cdot D \rho_\theta \zeta  \geq 0. \end{align*}
Then, we divide by $s$ and let $s\to 0$ to get  
\begin{align*}
&  -  \int_{\R^d}u_0 \psi(0) \rho_\theta - \int_0^T \int_{\R^d} u \psi  \rho_\theta\zeta' +
\int_0^T \int_{\R^d} \Bigl( \bar a (Du) \cdot D\psi -  f \psi - u \partial_t \psi  \Bigr) \rho_\theta\zeta \\
& \qquad  + \int_0^T \int_{\R^d} \psi \xi  \cdot D \rho_\theta \zeta  \geq 0. 
\end{align*}
Finally, letting $\zeta\to 1$ and $\theta\to 0$, so that $\zeta'\to 0$ and $\rho_\theta\to 1$  while $D\rho_\theta\to 0$ locally uniformly, we  get
\[
-\int_{\R^d}u_0 \psi(0) + \int_0^T \int_{\R^d} \Bigl( \bar a (Du) \cdot D\psi -  f \psi  - u \partial_t \psi  \Bigr)   \geq 0, 
\]
which, since $\psi$ is arbitrary, yields that  $u$ is a weak solution to \eqref{eq.AHlim} since $\psi$ is arbitrary.  
\smallskip

The proof of the $\P-$a.s. convergence of the family $(u^\ep)_{\ep>0}$ to $u$ in $L^2_{\rho_\theta}(\R^d\times [0,T])$ for any $\theta>0$ is now complete. Moreover,  in view of the estimates in \eqref{eq.Comega}, where $C_\theta^\ep$ converges in expectation, the $L^2$ convergence of $u^\ep$ to $u$ also holds in expectation. 
\smallskip

In the next subsections,  we prove  \eqref{limlim0}, \eqref{limlim1} and \eqref{limlim3} hold. 

\subsubsection{The proof of \eqref{limlim0}} The definition of $D\tilde \phi$ and $D\tilde \chi^\ep$ gives
\begin{align*}
& \int_0^T\int_{\R^d} a^\ep (D\tilde \phi(x,t)+D\tilde \chi^\ep(x,t), x,t) \cdot (D\tilde \phi(x,t)+D\tilde \chi^\ep) \rho_\theta(x)\zeta(t)  dxdt \\
& \qquad = \sum_k \int_{\hat Q_k}
a^\ep (p_k^\delta+D \chi^\ep(x,t; p_k^\delta), x,t) \cdot (p_k^\delta+D \chi^\ep(x,t; p_k^\delta))  \rho_\theta(x)\zeta(t)  dxdt.
\end{align*}
Since, in view of the choice of $p_k^\delta$ and of $\hat Q_k$, 
\eqref{eq.cvDchiep} and \eqref{eq.cvDchiepTER} hold, we get 
\begin{align*}
&\lim_{\ep\to 0} \int_0^T\int_{\R^d} a^\ep (D\tilde \phi(x,t)+D\tilde \chi^\ep(x,t), x,t) \cdot D\tilde \phi(x,t)) \rho_\theta(x)\zeta(t)  dxdt \\
& \qquad = \sum_k \int_{\hat Q_k}
\bar a(p_k^\delta) \cdot p_k^\delta \rho_\theta(x)\zeta(t)  dxdt = 
\int_0^T\int_{\R^d} \bar a(D\tilde \phi(x,t))\cdot  D\tilde \phi(x,t) \rho_\theta(x)\zeta(t) dxdt,
\end{align*}
which is  \eqref{limlim0}. 

\subsection{The proof of \eqref{limlim1}} The argument  is longer and more complicated.  
\smallskip

Using  again the piecewise structure of $D\tilde \chi^\ep$ and $D\tilde \phi$, we find  
\begin{align*}
&\int_0^T\int_{\R^d} a^\ep(Du^\ep(x,t),x,t)\cdot D\tilde \chi^\ep(x,t)\rho_\theta(x)\zeta(t)dxdt \\
& \qquad + \int_0^T\int_{\R^d} a^\ep(D\tilde \phi(x,t)+D\tilde \chi^\ep(x,t), x,t)) \cdot Du^\ep(x,t)\rho_\theta(x)\zeta(t)  dxdt\notag\\
&  = \sum_k \int_{\hat Q_k} \Bigl( a^\ep(Du^\ep(x,t),x,t)\cdot D\chi^\ep(x,t; p_k^\delta) \\
& \qquad\qquad\qquad +  a^\ep(p_k^\delta+D\chi^\ep(x,t;p_k^\delta), x,t)) \cdot Du^\ep(x,t) \Bigr)\rho_\theta(x)\zeta(t) dxdt. 
\end{align*}
Now we work separately  in each cube $\hat Q_k$. To simplify the notation, we denote by $\hat Q=Q_R(x_0)\times (t_0-T,t_0+T)$ a generic cube $\hat  Q_k$ and let $p= p_k^\delta$, $\chi=\chi(\cdot, \cdot;p)$, and  recall that 
 $R+T\leq \delta\leq 1$. 
 \smallskip
 
Note that \eqref{limlim1} follows, if we show that 
\begin{align}
\limsup_{\ep\to 0} \int_{\hat Q}& \Bigl( a^\ep(Du^\ep(x,t),x,t)\cdot D\chi^\ep(x,t) \notag\\
&+  a^\ep(p+D\chi^\ep(x,t), x,t)) 
 \cdot Du^\ep(x,t) \Bigr)\rho_\theta(x)\zeta(t) dxdt \label{ailzkejnsf} \\
& \qquad  
= \int_{\hat Q} \bar a(p) \cdot Du(x,t) \rho_\theta(x)\zeta(t) dxdt. \notag 
\end{align}

To proceed, we need to work with functions which are  compactly  supported in $\hat Q$. For this, we prove below that, for any $\delta'>0$, we can choose  $\psi\in C^\infty_c(\text{Int}(\hat Q))$ and $\ep_0>0$ such that 
\begin{align}
& \sup_{\ep\in (0,\ep_0)} \int_{\hat Q} \Bigl| \Bigl( a^\ep(Du^\ep(x,t),x,t)\cdot D\chi^\ep(x,t) \notag\\
& \qquad \qquad +  a^\ep(p+D\chi^\ep(x,t), x,t) \cdot Du^\ep(x,t)\Bigr)\rho_\theta(x)\zeta(t)\Bigr||1-\psi(x,t)| dxdt \label{alsrdgc} \\
& \qquad + \int_{\hat Q} |\bar a(p) \cdot Du(x,t) \rho_\theta(x)\zeta(t)| |1-\psi(x,t)| dxdt \leq \delta'. \notag
\end{align}
Then, we show that, if $\kappa:= \rho_\theta \zeta\psi$, then
\be \label{kujqhdbjsfnx}
 \lim_{\ep\to 0} \int_{\hat Q} \Bigl( a^\ep(Du^\ep,x,t)\cdot D\chi^\ep +  a^\ep(p+D\chi^\ep, x,t) \cdot Du^\ep \Bigr)\kappa dxdt 
=\int_{\hat Q} \bar a(p) \cdot Du \kappa dxdt. 
\ee
Once we know \eqref{kujqhdbjsfnx}, we can  combine   \eqref{ailzkejnsf} and  \eqref{alsrdgc} to get 
\begin{align*}
&\limsup_{\ep\to 0} \Bigl| \int_{\hat Q} \Bigl( a^\ep(Du^\ep(x,t),x,t)\cdot D\chi^\ep(x,t) +  a^\ep(p+D\chi^\ep(x,t), x,t)) \cdot Du^\ep(x,t) \Bigr)\rho_\theta(x)\zeta(t) dxdt \\
& \qquad \qquad - \int_{\hat Q} \bar a(p) \cdot Du(x,t) \rho_\theta(x)\zeta(t) dxdt\Bigr|\leq 2\delta', 
\end{align*}
which gives the result since $\delta'$ is arbitrary. 
\smallskip

We now prove \eqref{kujqhdbjsfnx}. Using   $\chi^\ep\kappa$ as a test function in \eqref{pde10}  $u^\ep\kappa$ as a test function in the equation satisfied by  $\chi^\ep$ we get 
$$
\int_{t_0-T}^{t_0+T} \lg \partial_t u^\ep , \chi^\ep\kappa\rg_{H^{-1},H^1} + \int_{\hat Q} a(Du^\ep)\cdot (D\chi^\ep \kappa+ D\kappa \chi^\ep)=\int_{\hat Q} f \chi^\ep \kappa.
$$
and 
$$
\int_{t_0-T}^{t_0+T} \lg \partial_t \chi^\ep , u^\ep \kappa\rg_{H^{-1},H^1} + \int_{\hat Q} a(p+D\chi^\ep)\cdot (Du^\ep  \kappa+ D\kappa u^\ep)=0,
$$
and, hence, after using an easy regularization argument, we find 
\begin{align*}
\int_{\hat Q} f \chi^\ep \kappa &= \int_{t_0-T}^{t_0+T} \lg \partial_t u^\ep , \chi^\ep\kappa\rg_{H^{-1},H^1} + \int_{t_0-T}^{t_0+T} \lg \partial_t \chi^\ep , u^\ep\kappa\rg_{H^{-1},H^1} \\
& \qquad \qquad + \int_{\hat Q} (a(Du^\ep)\cdot (D\chi^\ep \kappa+ D\kappa \chi^\ep)+ a(p+D\chi^\ep)\cdot (Du^\ep \kappa+ D\kappa u^\ep)) \\
&= -\int_{\hat Q} (u^\ep \chi^\ep) \partial_t \kappa \\
& \qquad + \int_{\hat Q} (a(Du^\ep)\cdot (D\chi^\ep \kappa+ D\kappa \chi^\ep)+ a(p+D\chi^\ep)\cdot (Du^\ep\kappa+ D\kappa u^\ep)) ,
\end{align*}
Recalling \eqref{eq.cvDchiep}, \eqref{alkejzred} and that, in view of \eqref{cvchiep}, $\chi^\ep \to 0$  in $L^2_{loc}$, we  
pass to the limit $\ep\to $ in the last equalities and get 
\begin{align*}
& \lim_{\ep}  \int_{\hat Q} (a(Du^\ep)\cdot D\chi^\ep+  a(p+D\chi^\ep)\cdot Du^\ep) \kappa + \bar a (p) D\kappa u =0
\end{align*}
An  integration  by parts then yields \eqref{kujqhdbjsfnx}.
\smallskip

To complete the proof, we show that it is possible to build $\psi$ with values in $[0,1]$  in such a way that \eqref{alsrdgc} holds. 
Indeed,   choose an increasing family of  cubes $(\hat Q_n)_{n\in \N}$  in ${\mathcal Q}$ (recall was  ${\mathcal Q}$ is defined in subsubsection \ref{def.Omega0E}) such that $|\hat Q\backslash \hat Q_n|\to 0$.   Then, given $\gamma>0$ to be chosen below,  in view of  \eqref{eq.cvDchiepBIS} and  for $n$ large enough, we have 
$$
\lim_{\ep} \int_{\hat Q\backslash \hat Q_n} |D\chi^\ep|^2 dxdt = \E\left[ \int_{\tilde Q_1}|D\chi|^2\right] |\tilde Q\backslash \tilde Q_n| \leq \gamma^2/2.
$$

Hence, there exists $\ep_0$ such that 
$$
\sup_{\ep\in (0, \ep_0]} \int_{\hat Q\backslash \hat Q_n} |D\chi^\ep|^2 dxdt\leq \gamma^2.
$$
Choose $\psi\in C^\infty_c(\text{Int}(\hat Q);[0,1])$ such that $\psi = 1$ in $\hat Q_n$. Then, for any $\ep\in (0,\ep_0]$,   
\begin{align*}
&  \int_{\hat Q} \Bigl| a^\ep(Du^\ep)\cdot D\chi^\ep \rho_\theta(x)\zeta(t)\Bigr||1-\psi(x,t)| dxdt \\
& \qquad \leq  \|\rho_\theta\zeta\|_\infty  \|a^\ep(Du^\ep)\|_{L^2(\hat Q)} \|D\chi^\ep |1-\psi|\|_{L^2(\hat Q)} \\
& \qquad \leq  \|\rho_\theta\zeta\|_\infty (\| a^\ep(0)\|_{L^2(\hat Q)}+C_0 \|Du^\ep\|_{L^2(\hat Q)}) \|D\chi^\ep |1-\psi|\|_{L^2(\hat Q)} \\
& \qquad \lesssim_\omega \|D\chi^\ep \|_{L^2(\hat Q\backslash \hat Q_n)} \lesssim_\omega \gamma;
\end{align*}
the dependence on $\omega$ is through the constants in \eqref{eq.Comega} and in \eqref{eq.choixomega3}. 
\smallskip

Treating the  other terms in \eqref{alsrdgc} similarly we obtain,  for $\gamma$ small enough, 
\begin{align*}
& \sup_{\ep\in (0,\ep_0)} \int_{\hat Q} \Bigl| \Bigl( a^\ep(Du^\ep(x,t),x,t)\cdot D\chi^\ep(x,t)  \\
& \qquad \qquad +  a^\ep(p+D\chi^\ep(x,t), x,t) \cdot Du^\ep(x,t) \Bigr)\rho_\theta(x)\zeta(t)\Bigr||1-\psi(x,t)| dxdt \notag\\
& \qquad + \int_{\hat Q} |\bar a(p) \cdot Du(x,t) \rho_\theta(x)\zeta(t)| |1-\psi(x,t)| dxdt \lesssim_\omega C\gamma\leq \delta' .\notag
\end{align*}

\subsubsection{The proof of \eqref{limlim3}} Note first that 
\begin{equation}\label{aomkesfd1}
\lim_{\sigma\to 0} \int_{\R^d} (\frac{u_0^2}{2}-u_0u^\sigma(0)) \rho_\theta  + \int_0^T \int_{\R^d} u(\frac{u}{2}-u^\sigma) \rho_\theta\zeta'  
= 
 \int_{\R^d} -\frac{u_0^2}{2}\rho_\theta  - \int_0^T \int_{\R^d} \frac{u^2}{2} \rho_\theta\zeta'.
\end{equation}
On the other hand,  the weak convergence, as $\sigma\to 0$,  of $\partial_t u^\sigma$ to $\partial_t u$ yields 
$$
\int_0^T \|\partial_t u^\sigma\|_{H^{-1}_{\rho_\theta}}^2dt\lesssim 1 . 
$$
Thus, as $\sigma\to 0$,
\begin{align*}
\left| \int_0^T \int_{\R^d} (u-u^\sigma) \partial_t u^\sigma \rho_\theta \zeta\right|& \leq
\left(\int_0^T \|\partial_t u^\sigma\|_{H^{-1}_{\rho_\theta}}^2dt\right)^{1/2} 
\left( \int_0^T  \|(u-u^\sigma)  \zeta\|_{H^1_{\rho_\theta}}^2 \right)^{1/2}\\ 
&\lesssim \left( \int_0^T  \|(u-u^\sigma) \|_{H^1_{\rho_\theta}}^2 \right)^{1/2} \to 0. 
\end{align*}
Therefore, 
\begin{align} 
& \lim_{\sigma\to 0} \int_0^T \int_{\R^d} u\partial_t u^\sigma \rho_\theta \zeta = \lim_{\sigma\to 0} \int_0^T \int_{\R^d} u^\sigma \partial_t u^\sigma \rho_\theta \zeta 
\notag \\ & \qquad
= \lim_{\sigma\to 0} - \int_{\R^d} \frac{(u^\sigma)^2(0)}{2} -  \int_0^T \int_{\R^d} \frac{(u^\sigma)^2}{2} \rho_\theta \zeta' 
=  - \int_{\R^d} \frac{u_0^2(0)}{2} -  \int_0^T \int_{\R^d} \frac{u^2}{2} \rho_\theta \zeta' . \label{aomkesfd2}
\end{align}
Combining \eqref{aomkesfd1} and \eqref{aomkesfd2} gives \eqref{limlim3}.



\section{The Homogenization of \eqref{FS}}\label{sec.homoFS}

We use the results of the two previous sections to study the behavior, as $\ep\to $, of \eqref{FS}.
\smallskip

We begin with the assumptions. As far the $(B^k)_{k\in \Z^d}$ and $A$ are concerned we assume \eqref{B} and \eqref{A}. 
\smallskip

We also assume 
\be\label{omega1}
\begin{cases}
 ( \Omega_1,  {\mathcal F}_1,  \P_1) \text{is a probability space  endowed with an ergodic} \\[1.2mm]
\text{ measure-preserving group of transformations  $\tau:\Z^d\times \R\times \Omega_1\to \Omega_1$,}
 \end{cases}
 \ee
and 
\be\label{F}
\begin{cases}
\mathcal A:\R^d\times\R^d\times \R \times \Omega_1\to \R^d \ \text{is a smooth and stationary in $( \Omega_1,  {\mathcal F}_1,  \P_1)$}\\[1.2mm]
\text{vector field, which is strongly monotone and Lipschitz continuous}\\[1.2mm]
\text{in the first variable, uniformly with respect to the other variables;}
\end{cases}
\ee
note that the family $(B^k)_{k\in \Z^d}$ and the vector field  $\mathcal A$ are defined in different probability spaces.
\smallskip

%
%
Finally, for  the random environment we assume that 
\be\label{omega2}
(\Omega, {\mathcal F}, \P) \ \text{is  the product probability space of $(\Omega_0,\mathcal F_0,\P_0)$ and $( \Omega_1,  {\mathcal F}_1,  \P_1)$,}
\ee
that is, $\Omega=\Omega_0\times \Omega_1, {\mathcal F}= {\mathcal F}_0\otimes {\mathcal F}_1$ and $\P=\P_0\otimes \P_1$.
\smallskip

We continue making precise the meaning  of a solution of \eqref{FS}. A field $U^\ep$ solves \eqref{FS} if 
\be\label{UepVepWep}
U^\ep_t(x,\omega)= \ep V_{\frac{t}{\ep^2}}(\dfrac{x}{\ep},\omega_0) +  W^\ep(x,t,\omega)= V^\ep_{t}(x,\omega_0) +  W^\ep(x,t, \omega),
\ee
%
with $V$ and $W^\ep$ solving  respectively \eqref{l1} and 
\be\label{W}
\ds \partial_t W^\ep_t = {\rm div} \left( \hat a^\ep (D W^\ep_t,x,t,\omega)\right) \ \  {\rm in} \ \ \R^d\times (0,+\infty) \quad W^\ep_0= u_0 \ \ \text{in} \ \ \R^d,
\ee
where  
\be\label{W1}
\hat a^\ep (p,x,t,\omega) = \hat a (p,\dfrac{x}{\ep}, \dfrac{t}{\ep^2},\omega), %
\ee
and
\be\label{W2}
 \hat a (p,x,t,\omega)= 
\mathcal A( p+DV_{t}(x,\omega_0), x,t,\omega_1)-DV_{t}(x,\omega_0). 
\ee
Note that $\hat a$ is strongly monotone and Lipschitz continuous in the first variable, uniformly with respect to the other variables, and satisfies \eqref{Hypa3} (thanks to Lemma \ref{lem.boundDV}). 
\smallskip
 
We say that $W^\ep:\R^d\times [0,T]\times \Omega\to \R $ is a solution of \eqref{W},
if  it  is measurable in $\omega$ for each $(x,t)$,   $W^\ep_\cdot(\cdot, \omega) \in L^2([0,T], H^1_{\rho_\theta})\cap C^0([0,T], L^2_{\rho_\theta})$ $\P-$a.s. with $\rho_\theta$ defined in \eqref{def.rhothetaBIS}, and it  satisfies \eqref{W} 
in the sense of distributions. It is easily checked that such a solution exists and is unique.


\begin{thm} \label{main.homogenization} Assume \eqref{B}, \eqref{A}, \eqref{omega1}, and \eqref{F}. 
Then there exists a strongly monotone and Lipschitz continuous vector filed  $\bar a :\R^d \to \R^d$ 
such that, for any $u_0\in L^2(\R^d)$, the solution $U^\ep$ of \eqref{FS} converges to the solution of the homogenized problem 
\be\label{eq.limiteq}
\partial_t \ol u= {\rm div} (\bar a(D\ol u)) \ \ {\rm in} \ \R^d\times (0,\oo) \quad \ol u(\cdot,0)=u_0 \ \ \text{in} \ \ \R^d,
\ee
in the sense that, for any $T>0$, 
$$
\lim_{\ep\to 0} \E\left[ \int_0^T \int_{\R^d}  | U^\ep_t(x)-u(x,t)|^2\rho_\theta(x)  dxdt\right]= 0,
$$
 where $\rho_\theta(x)= \exp\{-\theta(1+|x|^2)^{1/2}\}$. 
\end{thm}

The proof is a combination of the  results of the previous sections. 
\smallskip

The first step consists in replacing the non-stationary in time process $DV^\ep$ by the space-time stationary random field   $Z$  constructed in  Theorem \ref{thm:main}. To keep the notation in the statement simpler,  we introduce the maps $\tilde a^\ep$ and $\tilde a$ which are defined as
\be\label{W3}
\tilde a^\ep (p,x,t,\omega) = \tilde a (p,\dfrac{x}{\ep}, \dfrac{t}{\ep^2},\omega), %
\ee
and
\be\label{W4}
 \tilde a (p,x,t,\omega)= 
\mathcal A( p+Z_{t}(x,\omega_0), x,t,\omega_1)-Z_{t}(x,\omega_0).
\ee

\begin{lem}\label{lem.uhzbqensd}  Assume \eqref{B}, \eqref{A}, \eqref{omega1}, and  \eqref{F}, 
and let   $W^\ep$ and  $\tilde W^\ep$ be   respectively  solutions of \eqref{W} with $\hat a^\ep$ as in \eqref{W1} and 
\be\label{eq.tildeW}
 \partial_t \tilde W^\ep =  \div \Big(\tilde a^\ep (\tilde W^\ep, x, t, \omega)\Big) \ \ \text{in} \ \ \R^d\times (0,\oo) \ \ \  \tilde W^\ep_0= u_0 \ \ \text{in} \ \ \R^d,
 \ee
 with $\tilde a^\ep$ given by \eqref{W3}. 
 Then, for any $\theta>0$,  
\be\label{uhzbqensd}
\lim_{\ep\to 0} \sup_{t\in [0,T]} \E\left[ \int_{\R^d} |W^\ep (x,t)-\tilde W^\ep (x,t)|^2\rho_\theta(x) dx\right] =0.
 \ee

\end{lem} 

\begin{proof} Let  $V^\ep(x,t))= DV_{\frac{t}{\ep^{2}}}(\dfrac{x}{\ep})$ and  $Z^\ep_t(x)= Z_{\frac{t}{\ep^{2}}}(\dfrac{x}{\ep})$.
\smallskip

Using the strong monotonicity and Lipschitz continuity of $\mathcal A$  as well as \eqref{W1} and \eqref{W2} 
 we find, after some routine calculations,  that, for some constants $C>0$,  
\begin{align*}
& \frac{d}{dt} \E\left[\int_{\R^d}  (W^\ep_t-\tilde W^\ep_t)^2\rho_\theta dx\right] \leq 
-  \E\left[\int_{\R^d}   |D(W^\ep_t-\tilde W^\ep_t)|^2\rho_\theta  dx\right] \\
& \qquad +C \E\left[\int_{\R^d}  (W^\ep_t-\tilde W^\ep_t)^2\rho_\theta  dx\right]
+C \E\left[\int_{\R^d}   |DV^\ep_t- Z^\ep_t |^2\rho_\theta  dx\right].
\end{align*}
Since  $DV$ and $Z$ are stationary in space, we find
$$
\E\left[\int_{\R^d}   |DV^\ep_t- Z^\ep_t |^2 \rho_\theta dx\right] \leq 
C_\theta \E\left[ \int_{Q_1}|DV_{\ep^{-2}t}(x)-Z_{\ep^{-2}t}(x)|^2 dx \right],
$$
with the right hand side  bounded and converging, in view of \eqref{attractor},  to $0$ for $t>0$.
\smallskip

We conclude using  Gronwall's inequality. 

\end{proof}

\begin{proof}[The proof of Theorem \ref{main.homogenization}]
It now remains to show that  \eqref{FS} homogenizes. 
\smallskip

On $\Omega$ we define the ergodic measure preserving group $\tilde \tau :\Z^d\times \R\times \Omega\to \Omega$ by 
$$
\tilde \tau_{k,s} \omega= (\omega_0^{l+k}(s+\cdot), \tau_{k,s}\omega_1)
$$
for any
$
\omega=(\omega_0,\omega_1)= ((\omega_0^l)_{l\in \Z^d}, \omega_1)\in \Omega=(C^0(\R,\R^d))^{\Z^d} \times \Omega_1. 
$
\smallskip

Set 
$$
a(p,x,t, \omega)= \mathcal A( p+ Z_t(x,\omega_0), x,t,\omega_1) - Z_t (x,\omega)
$$
and note that  $a$ satisfies \eqref{T} and \eqref{a2}.
\smallskip

Then, in view of  Theorem~\ref{thm.homo},  the vector field $\ol a$  is strongly monotone and Lipschitz continuous 
and the solution $\tilde W^\ep$ of \eqref{eq.tildeW} converges, for all $\theta>0$,  $\P-$a.s. in $L^2_{\rho_\theta}(\R^d\times [0,T])$  and in $L^2_{\rho_\theta}(\R^d\times [0,T]\times \Omega)$  to the solution $\ol u$ of  \eqref{eq.limiteq}. 
\smallskip

Finally we return  to $U^\ep$. In view of \eqref{UepVepWep}, for any $\theta>0$,  we have
\begin{align*}
& \E\left[ \int_0^T\int_{\R^d} |U^\ep_t(x)- u(x,t)|^2\rho_\theta(x) dx dt \right] \\ 
& \leq 2 \E\left[ \int_0^T\int_{\R^d} |\ep V_{\ep^{-2}t}(\ep^{-1}x)|^2\rho_\theta(x)dx dt \right]
+ 2\E\left[ \int_0^T\int_{\R^d} |W^\ep_t(x)- \tilde W^\ep_t(x) |^2\rho_\theta(x)dx dt \right] \\
& \qquad +2 \E\left[ \int_0^T\int_{\R^d} |\tilde W^\ep_t(x)- u(x,t)|^2\rho_\theta(x)dx dt \right]. 
\end{align*}
In view of \eqref{ineq.timecontBIS}, Lemma \ref{lem.uhzbqensd} and Theorem \ref{thm.homo}, the right hand side of the inequality above tends to $0$ as $\ep\to 0$.

\end{proof}

\appendix
\section{}

We summarize here with proofs  results about stationary gradients, which are needed in the paper. Some of them appear in the literature  in different structures and with stronger assumptions.
\smallskip 

The following  is classical in the literature (see for instance the proof of Theorem 5.3 of \cite{KoVa}). We give a proof here because the environment has not exactly the same structure as in \cite{KoVa} and the maps here have lower regularity in time. 
 
\begin{lem} \label{append.lem2} Assume that  $(\Omega,{\mathcal F},\P)$ be a probability space endowed with an ergodic group of measure preserving maps $\tau:\Z^d\times \R\times \Omega\to \Omega$, and, for $i=1,\ldots,d$ and $t \in\R$, let  
${\mathcal G}_i$ and ${\mathcal G}_t$ be respectively the $\sigma-$algebra of sets $A\in {\mathcal F}$ such that, for any $k\in \Z$, 
 $\P[A\Delta (\tau_{(ke_i,0)}A)]=0$,  
and  the $\sigma-$algebra of sets $A\in {\mathcal F}$ such that, for any $s\in \R$,  $\P[A\Delta (\tau_{(0,s)}A)]=0$.
If $u:\R^d\times\R \times \Omega\to \R$ has space-time stationary weak derivatives $Du$ and $\partial_t u$ 
 such that 
\begin{align*}
&\E[\int_{\tilde Q_1}|Du|^2]<+\infty,  \ \ \E[\int_{\tilde Q_1}Du]=0, \ \ \E[\int_{0}^{1}\|\partial_t u(\cdot,t)\|_{H^{-1}(Q_1)}^2dt]<+\infty,\\
& \E[\int_{0}^{1}\lg \partial_t u(\cdot,t),1\rg_{H^{-1}(Q_1),H^1(Q_1)}dt]=0 \ \  \text{and} \ \ 
\int_{Q_1}u dx =0 \ \  \P-\text{a.s.},
\end{align*}
then, for any $i=1,\dots, d$ and any $(z,t)\in \R^d\times  \R$, 
$$
\E\left[ \int_{Q_1} \partial_{x_i}u(\cdot+z,t)\ \Bigl|\ {\mathcal G}_i\right]=0 \ \ \text{and} \ \ \E\left[ \lg \partial_t u(\cdot +z,t), 1\rg_{H^{-1}(Q_1),H^1(Q_1)}\ \Bigl|\ {\mathcal G}_t\right] =0.
$$
\end{lem} 

\begin{proof} To fix the ideas we prove the result for $i=1$.
\smallskip

Fix $(z,s)\in \R^{d}\times \R$ and let $\xi:\R^d\times \R\times \Omega\to \R$ be bounded, stationary, and ${\mathcal G}_1-$measurable. 
For any  $n\in \N$ large, we have 
\begin{align*}
& \int_{\tilde Q_1} (u(x+ne_1+z,t+s)-u(x,t))\xi(x,t)dxdt\\
&  = \sum_{l=0}^{n-1}\int_{\tilde Q_1}\int_0^1\partial_{x_1}u(x+le_1+re_1+z,t+s)\xi(x,t)dxdtdr\\ 
& + \int_{\tilde Q_1}(u(x+z,t+s)-u(x,t))\xi(x,t)dxdt.
\end{align*}
It follows from the  stationarity of $\partial_{x_1}u$ and the  ${\mathcal G}_1-$measurability of $\xi$ that 
\begin{align*}
&\E\left[ \int_{\tilde Q_1} (u(x+ne_1+z,t+s)-u(x,t))\xi(x,t)dxdt\right] \\
& =  n\E\left[ \int_0^1\int_{\tilde Q_1}\partial_{x_1}u(x+re_1+z,t+s)\xi(x,t)dxdtdr  \right] \\ 
& \qquad\qquad +\E\left[  \int_{\tilde Q_1}(u(x+z,t+s)-u(x,t))\xi(x,t)dxdt \right]\\
& = n\E\left[\int_{\tilde Q_1}\partial_{x_1}u(x+z,t+s)\xi(x,t)dxdt  \right] + 
\E\left[  \int_{\tilde Q_1}(u(x+z,t+s)-u(x,t))\xi(x,t)dxdt\right],
\end{align*}
the last two lines following  from the $\Z-$periodicity of $s\to \E\left[\partial_{x_1}u(x+se_1+z,t+s)\xi(x,t)dxdt \right] $. 
\smallskip

Hence 
\begin{align*}
\underset{n\to \infty}\lim \dfrac{1}{n} \E\Big[ \int_{\tilde Q_1} & (u(x+ne_1+z,t+s)-u(x,t))\xi(x,t)dxdt \Big]\\
 & = \E \Big[\int_{\tilde Q_1}\partial_{x_1}u(x+z,t+s)\xi(x,t) dxdt  \Big] .
\end{align*}

On the other hand, 
\begin{align*}
& \E\left[\int_{\tilde Q_1} (u(x+ne_1+z,t+s)-u(x,t))\xi(x,t)dxdt \right] =\\
&  \qquad \qquad \qquad \qquad  \E\left[\int_{\tilde Q_1} (u(x+ne_1,t)-u(x,t))\xi(x,t)dxdt \right] \\ 
&\qquad \qquad \qquad \qquad + \E\left[\int_{\tilde Q_1} (u(x+ne_1+z,t+s)-u(x+ne_1+z,t))\xi(x,t)dxdt \right]\\
&\qquad \qquad \qquad \qquad + \E\left[\int_{\tilde Q_1} (u(x+ne_1+z,t)-u(x+ne_1,t))\xi(x,t)dxdt \right].
\end{align*}
The goal is to divide by $n$ and let $n\to +\infty$. The left-hand side and the first term in the right-hand side have a limit given by the previous equality. 
\smallskip

We  show next that the two remaining terms  after divided by $n$ tend to $0$.
\smallskip

 In order to use the time regularity of $u$, we need to regularize in space the indicatrix function of $Q_1$. Let $\zeta_\delta \in C^\infty_c(Q_1)$ with $\|1-\zeta_\delta \|_{L^{2}(Q_1)}\leq \delta$. 
 \smallskip
 
 Then, using the stationarity of $\partial_t u$ and the fact that $\xi$ is ${\mathcal G}_1-$measurable, we find 
\begin{align*}
& \E\left[\int_{\tilde Q_1} (u(x+ne_1+z,t+s)-u(x+ne_1+z,t))\zeta_\delta(x) \xi(x,t)dxdt \right] \\ &
= 
\E\left[\int_{-1/2}^{1/2} \int_0^{s} \lg \partial_t u(\cdot+ne_1+z,t+s'), \zeta_\delta \xi(\cdot,t)\rg_{H^{-1},H^1} ds'dt \right]\\
&= \E\left[\int_{-1/2}^{1/2} \int_0^{s} \lg \partial_t u(\cdot+z,t+s'), \zeta_\delta \xi(\cdot,t)\rg_{H^{-1},H^1} ds'dt \right], 
\end{align*}
Thus,  
\begin{align*}
& \E\left[\int_{\tilde Q_1} (u(x+ne_1+z,t+s)-u(x+ne_1+z,t))\zeta_\delta(x) \xi(x,t)dxdt \right] \\ 
& \qquad =\E\left[\int_{\tilde Q_1} (u(x+z,t+s)-u(x+z,t))\zeta_\delta(x) \xi(x,t)dxdt \right] , 
\end{align*}
and, after  letting $\delta\to 0$, 
\begin{align*}
& \E\left[\int_{\tilde Q_1} (u(x+ne_1+z,t+s)-u(x+ne_1+z,t)) \xi(x,t)dxdt \right] \\ 
& \qquad = \E\left[\int_{\tilde Q_1} (u(x+z,t+s)-u(x+z,t)) \xi(x,t)dxdt \right] . 
\end{align*}
Similarly, using the stationarity of $Du$, we get 
$$
\E\left[\int_{\tilde Q_1} (u(x+ne_1+z,t)-u(x+ne_1,t))\xi(x,t)dxdt \right]
=
\E\left[\int_{\tilde Q_1} (u(x+z,t)-u(x,t))\xi(x,t)dxdt \right]. 
$$
It follows that, for any $(z,s)\in \R^d\times \R$ and  any ${\mathcal G}_1-$measurable $\xi$,
$$
\E\left[ \int_{\tilde Q_1}(\partial_{x_1}u(x+z,t+s)-\partial_{x_1}u(x,t))\xi(x,t)dxdt \right]=0
$$
Hence,  the map 
\[(z,s)\to \E\left[ \int_{\tilde Q_1}(\partial_{x_1}u(x+z,t+s)dxdt\ |\ {\mathcal G}_1\right]\] 
is $\P-$a.s constant.  Since it is also  stationary in an ergodic environment, it must also be constant in  $\omega$ and, as it has a zero expectation, it has to be equal to $0$. 

The proof of the time derivative follows is similar and, hence, we omit it.

\end{proof}

We  discuss next the sublinearity of maps with stationary derivatives.
%
%

\begin{lem} \label{lem.appendix} Let $(\Omega,{\mathcal F},\P)$ and $u:\R^{d+1} \times \Omega\to \R$ be as in Lemma \ref{append.lem2}. 
Then, $\P-$a.s. and in expectation, 
$$
\underset{R\to \oo} \lim R^{-(d+2)} \int_{Q_R} |u(x,0)|^2 dxdt =0 \ \  {\rm and}\ \  \underset{R\to \oo} \lim R^{-(d+3)} \int_{\tilde Q_R} |u(x,t)|^2 dxdt.
$$
\end{lem}

The above result can also be formulated  as follows. Let $u^\ep(x,t,\omega)= \ep u(x/\ep, t/\ep,\omega)$. Then, for any fixed $R>0$, $\P-$a.s. and in expectation, 
$$
\underset{\ep\to 0}\lim \int_{ Q_R} |u^\ep(x,0)|^2dxdt =0\ \  {\rm and}\ \ 
\underset{\ep\to 0}\lim  \int_{\tilde Q_R} |u^\ep(x,t)|^2dxdt =0.
$$
Note that, here, the scaling is hyperbolic in contrast with what we did throughout the paper.

\begin{proof}
In view of Lemma~\ref{append.lem2},  we can apply Theorem 5.3 of \cite{KoVa} to the map $x\to u(x,0)$ to infer that, for any $R>0$ and $\P-$a.s.  
\be\label{lzkejqnsf}
\lim_{\ep\to 0} \int_{Q_R} |u^\ep(x,0)|dx = 0. 
\ee
In \cite{KoVa}, the problem  is stationary with respect to any (space) translation, while here the problem is $\Z^d-$ stationary. However, a careful inspection of the proof of Theorem 5.3 in \cite{KoVa} shows that the result still holds in our setting, the key point of the proof in \cite{KoVa} being precisely the statement in Lemma \ref{append.lem2}.
\smallskip
 
Let $\xi_R\in C^\infty_c(\R^d; [0,1])$ be such that $\xi_R= 1$ in $Q_{R}$, $\xi_R=0$ in $\R^{d+1}\setminus Q_{R+1}^c$ and  $\|D\xi_R\|_\infty\leq 2$. Then, after an integration by parts in time, we have 
\begin{align*}
\int_0^{R/2} \int_{Q_R} u^\ep(x,t) \xi_R(x) dxdt  = \int_{Q_R} u^\ep(x,0) \xi_R(x) dxdt  - \int_0^{R/2} (1-t) \lg \partial_t u^\ep(\cdot,t),  \xi_R\rg_{H^{-1}, H^1} dt.  
\end{align*}
In view of  \eqref{lzkejqnsf}, the first term in the right-hand side tends to $0$ as $\ep\to 0$, while, since $q(k,s,\omega)=\lg \partial_t u(\cdot, t), 1\rg_{H^{-1}(Q_1),H^1(Q_1)}$ is stationary,   the ergodic theorem
also implies that the second term in the right-hand side has a  $\P-$a.s limit., which again does not depend on $\omega$ and, therefore,  has to be zero since  $\E[\lg \partial_t u,1\rg_{H^{-1},H^1}]=0$. 
\smallskip

It follows that, $\P-$a.s.,
\begin{align*}
\limsup_{\ep\to 0} \int_0^{R/2} \int_{Q_R} u^\ep(x,t) \xi_R(x) dxdt  = 0. 
\end{align*}
Applying  similar arguments on the time interval $[-T,0]$, we also find that, $\P-$a.s.,
\be\label{lsekjcv}
\lim_{\ep\to 0} \int_{\tilde Q_R} u^\ep(x,t)  \xi(x) dx dt =0.
\ee

Next,  we claim that there exists a constant $C$ such that,  $\P-$a.s.,
$$
\limsup_{\ep\to 0} \int_{\tilde Q_R}(u^\ep(x,t))^2dxds \leq C. 
$$
Indeed, set
$$
\lg u^\ep \rg_{\xi_R} := (\int_{\tilde Q_R} \xi_R(x)dxdt)^{-1} \int_{\tilde Q_R} u^\ep(x,t)\xi_R(x)dxdt, 
$$
and observe that a minor  generalization of the classical Poincar\'{e} 's inequality yields, for some $C_R$ which depends on $\xi_R$, 
\begin{align*}
&\int_{\tilde Q_R} (u^\ep(x,t))^2dxdt  \leq 2 \int_{\tilde Q_R}(u^\ep(x,t)-\lg u^\ep\rg_{\xi_R})^2 dx +2 R^{d+1}  \lg u^\ep\rg_{\xi_R}^2\\ 
& \leq  C_R [ \int_{\tilde Q_R}|Du^\ep(x,t)|^2 dxdt +\int_{-R/2}^{R/2} \|\partial_t u^\ep(\cdot, t)\|_{H^{-1}(Q_R)}^2 dt]
+ 2 R^{d+1}  \lg u^\ep\rg_{\xi_R}^2,
\end{align*}
Then, the ergodic Theorem and  \eqref{lsekjcv}, give that, $\P-$a.s., 
$$
\limsup_{\ep\to 0} \int_{\tilde Q_R} (u^\ep(x,t))^2dx \leq  C_R \E\left[   \int_{\tilde Q_1}|Du(x,t)|^2 dxdt + \int_{-1/2}^{1/2}\|\partial_t u(t)\\|_{H^{-1}(Q_1)}^2dt \right]. 
$$

To summarize,  we have shown  that there exists $\Omega_0\subset \Omega$ on which,  for every $R>0$ and $\P-$a.s., 
\be
\label{m20}
\text{the family $(u^\ep)_{\ep>0}$ is  bounded respectively in  $L^2([-R,R], H^1(Q_R))$,}
\ee
and
\be\label{ma21}
\text{the family $(\partial_t u^\ep)_{\ep>0}$ is  bounded   in  $L^2([-R,R], H^{-1}(Q_R))$.}
\ee

It follows that, for any $\omega\in \Omega_0$, the family   $(u^\ep)_{\ep>0}$ is relatively compact in $L^2_{loc}(\R^{d+1})$. 
\smallskip

Let $(u^{\ep_n})_{n\in \N}$ be a  sequence which converges in $L^2_{loc}(\R^{d+1})$ to some $u\in L^2([-R,R], H^1(Q_R))$ with  $\partial_t u$ in $L^2([-R,R], H^{-1}(Q_R))$. Since, as $\ep \to 0$,  $Du^\ep \rightharpoonup 0$ and $\partial_t u^\ep \rightharpoonup 0$, $u$ is  a constant, which, in view of  \eqref{lsekjcv},  must be zero. 
\smallskip

It follows that, as $n\to \infty$ and in  $L^2_{loc}(\R^{d+1})$, $u^{\ep_n} \to 0$, and, therefore that, as $\ep \to 0$ and $\P-$a.s.,  $u^\ep\to 0$ in $L^2_{loc}(\R^{d+1})$, and, by the estimate above, in expectation. 
\smallskip

The claim for $u^\ep(\cdot,0)$ follows similarly and with a  simpler argument, hence, we omit it.

\end{proof}

\bibliographystyle{siam}

\end{document}